\documentclass{amsart}
\usepackage{bbm}
\usepackage{amsfonts}
\usepackage{amsmath}
\usepackage{amsthm,mathrsfs}
\usepackage{amsfonts,amssymb}
\usepackage{float}
\usepackage{pdfsync}
\usepackage{graphicx}
\graphicspath{{relative_bounded_cohomology/}}
\usepackage{caption}
\usepackage{color}
\usepackage{cancel,ulem}
\usepackage{tikz}
\usepackage{hyperref}
\allowdisplaybreaks[4]

\newtheorem{theorem}{Theorem}[section]
\newtheorem{proposition}[theorem]{Proposition} 

\newtheorem{lemma}[theorem]{Lemma}
\theoremstyle{definition}
\newtheorem{definition}[theorem]{Definition}
\newtheorem{question}[theorem]{Question} 
\newtheorem{example}[theorem]{Example}
\newtheorem{remark}[theorem]{Remark}
\theoremstyle{remark}
\theoremstyle{corollary}
\newtheorem{corollary}[theorem]{Corollary}
\theoremstyle{remark}

\newtheorem*{claim}{Claim}
\numberwithin{equation}{section}
\DeclareMathOperator{\ax}{Ax}

\def\diam{{\mathop\mathrm{\,diam\,}}}

\def\scl{\mathrm{scl}}
\def\ax{\mathrm{Ax}}
\def\cl{\mathrm{cl}}
\def\Mod{\mathrm{Mod}}

\def\F{\mathcal F}
\def\C{\mathcal C}
\def\P{\mathcal P}
\def\G{\mathcal G}
\def\R{\mathbb R}
\def\X{\mathbb X}
\def\Z{\mathbb Z}
\def\N{\mathbb N}
\def\llangle{\langle\langle}
\def\rrangle{\rangle\rangle}

\newcommand{\revise}[1]{{\color{black}{#1}}}
\newcommand{\rev}[1]{{\color{black}{#1}}}

\begin{document}

\title{Relative Bounded Cohomology on Groups with Contracting Elements}

\author{Zhenguo Huangfu}
\address{Institute of Mathematical Sciences\\ ShanghaiTech University\\
 Shanghai 201210, China P.R. }
\email{huangfuzhg@shanghaitech.edu.cn}

\author{Renxing Wan}
\address{School of Mathematical Sciences,  Key Laboratory of MEA (Ministry of Education) \& Shanghai Key Laboratory of PMMP,  East China Normal University, Shanghai 200241, China P. R.}
\email{rxwan@math.ecnu.edu.cn}


\begin{abstract}
    Let $G$ be a countable group acting properly on a metric space with contracting elements and $\{H_i:1\le i\le n\}$ be a finite collection of Morse subgroups in $G$. We prove that each $H_i$ has infinite index in $G$ if and only if the relative second bounded cohomology $H^{2}_b(G, \{H_i\}_{i=1}^n; \mathbb{R})$ is infinite-dimensional. In addition, we also prove that for any contracting element $g$, there exists $k>0$ such that $H^{2}_b(G, \langle \langle g^k\rangle \rangle; \mathbb{R})$ is infinite-dimensional. Our results generalize a theorem of Pagliantini-Rolli for finite-rank free groups and yield new results on the (relative) second bounded cohomology of groups.

\end{abstract} 

\maketitle

\section{Introduction}
\subsection{(Relative) bounded cohomology}
Bounded cohomology  was introduced by Johnson and Trauber in the context of Banach algebra and developed into a comprehensive and rich theory by Gromov in his seminal paper ``Volume and bounded cohomology'' \cite{Gro82}. Since then, it has become a fundamental tool in several fields, most notably in the study of the geometry of manifolds. Many properties, in particular geometric ones, \rev{of a group, can be} characterized by its bounded cohomology. In particular, the notion of bounded cohomology allows us to retrieve certain negatively curved features of the group.

The theory of quasimorphisms has been extensively exploited to study the second bounded cohomology of a group. To be precise, a \textit{quasimorphism} on a group $G$ is a map $\phi: G\rightarrow \mathbb{R}$ such that

\begin{center}
$\sup\limits_{g,h \in G}|\phi(gh)-\phi(g)-\phi(h)|<\infty$.
\end{center}

\noindent That is to say, a quasimorphism is \revise{locally close to} a genuine homomorphism from the group to \rev{$\mathbb{R}$}. We denote by $\rm{QM}(G)$ the $\mathbb R$-vector space of quasimorphisms. It \revise{follows from the definitions} that the coboundary of a quasimorphism is a bounded 2-cocycle and therefore there is a linear map $$\rm{QM}(G)\to H^2_b(G;\mathbb R).$$
Moreover, the image of this map is the kernel of the comparison map $H^2_b(G;\mathbb R)\to H^2(G;\mathbb R)$ induced by the inclusion of bounded cochains into ordinary cochains \revise{\cite[Theorem 2.50]{Cal09}}.

The first example of this approach is using Brooks counting quasimorphism \cite{Bro81} to show that the rank-2 free group $F_2$ has infinite dimensional second bounded cohomology.

More generally, D. Epstein and K. Fujiwara \cite{EF97} showed that a non-elementary hyperbolic group has infinite dimensional second bounded cohomology. \revise{This was proved by using a modified version of Brooks counting quasimorphism.}  Later, Fujiwara used this generalization to obtain the same conclusion regarding the dimension of the second bounded cohomology of  a group acting properly on a hyperbolic space \cite{Fuj98}.

When dealing with manifolds with boundary, we naturally consider the relative homology and cohomology. Similarly, in group theory, we often investigate the subgroups in order to \revise{explore} properties of the ambient group. M. Gromov provided a definition of relative bounded cohomology between pairs of topological spaces and also pairs of groups in \cite{Gro82}. This leads to many applications in geometry and topology and other fields as well.

In retrospect, absolute bounded cohomology of many classes of groups are already known to be infinite dimensional, as demonstrated in the following examples:

\begin{enumerate}
    \item
    non-elementary Gromov hyperbolic groups \cite{EF97};
    \item
    groups with infinitely many ends \cite{Fuj00};
    \item
    groups admitting a non-elementary proper discontinuous action on a Gromov hyperbolic space \cite{Fuj98};
    \item
    groups admitting a non-elementary weakly proper discontinuous action on a Gromov hyperbolic space \cite{BF02} (see also \cite{Ham08});
    \item
    groups admitting a non-elementary weakly proper discontinuous action on a CAT(0) space which contains a rank one isometry \cite{BF09};
    \item
    acylindrically hyperbolic groups \cite{HO13}.
\end{enumerate}

\revise{Gromov developed the theory of bounded cohomology in order to compute simplicial volume of manifolds. In particular, the understanding of
simplicial volume of manifolds with non-empty boundary could be increased by studying bounded cohomology of pairs of space and pairs of groups, which is the so-called \textit{relative bounded cohomology}. However,} very few results on relative bounded cohomology are known. The first computation of the relative bounded cohomology of a specific class of groups is given by C. Pagliantini and P. Rolli in \cite{PR15}. \revise{Besides, it was proved independently by Bucher-Burger-Frigerio-Iozzi-Pagliantini-Pozzetti \cite{BBFIPP14} and Kim-Kuessner \cite{KK15} that the bounded cohomology of a CW-complex $X$ relative to an amenable subcomplex $Y$ is isometrically isomorphic to the absolute bounded cohomology of $X$; in particular, in this situation, $H^{\ast}_b(X,Y; \R)$ is infinite-dimensional whenever $X$ fits into the list given in the previous paragraph.
Another relevant result is given by Franceschini \cite{Fra18}, who proves that if $(G, H)$ is a relatively hyperbolic pair, then the comparison map $H^{k}_b(G,H; V)\to H^{k}(G,H; V)$ is surjective for every $k\ge 2$ and any bounded $G$-module $V$.}

\subsection{Contracting element} The purpose of this paper is to generalize the existing results on the second bounded cohomology mentioned above  to the relative version. Note that the examples of groups described above all satisfy the contracting properties outlined below.

Let $X$ be a geodesic metric space and $A$ be a \revise{closed} subset of $X$. For a constant $C\ge 0$, we denote by $N_C(A)$ the \rev{open} $C$-neighborhood of $A$ in $X$. Let $\pi_A: X \rightarrow \revise{2^A, x\mapsto \pi_A(x)=\{y\in A: d(x,y)=d(x,A)\}}$ be the map \revise{given by the closest point projection}. We say that $A$ is \textit{$C$-contracting} for $C \geq 0$ if $\diam(\pi_A(\gamma))\le C$ for any geodesic (segment) $\gamma$ with $ \gamma\cap N_C(A)=\varnothing$. In fact, this notion of contracting is equivalent to the usual one: $\diam(\pi_A(B)) \leq C'$ for any metric ball $B$ disjoint from $A$; \rev{see \cite[Corollary 3.4]{BF09} for a proof}. For an isometric group action of a group $G$ on $X$, an element $g\in G$ is called \textit{contracting}, if some (or equivalently, any) orbit of $\langle g\rangle$ is a contracting quasi-geodesic.

The prototype of a contracting element is a loxodromic isometry on a Gromov hyperbolic space, but many more examples are known to be contracting:
\begin{itemize}
    \item rank-1 elements in CAT(0) groups acting on a CAT(0) space, cf. \cite{Bal12,BF09};
    \item hyperbolic elements in groups with non-trivial Floyd boundary (e.g., relatively hyperbolic groups) acting on their Cayley graph with respect to a generating set, cf. \cite{Dru05,Osi04,Yan14};
    \item certain infinite-order elements in graphical small cancellation groups acting on their Cayley graph with respect to a generating set, cf. \cite{ACGH16};
    \item pseudo-Anosov elements in mapping class groups acting on the Teichm\"uller space  equipped with Teichm\"uller metric, or on the curve complex, cf. \cite{Far12,Prz17,Thu79}.
\end{itemize}

In this paper, a group is called \textit{non-elementary} if it is not virtually cyclic. Given an isometric group action on a metric space, a subgroup is called \textit{Morse} if some (or equivalently, any) orbit of this subgroup is weakly quasi-convex (cf. Definition \ref{Def: MorseSubgp}). Now, we state our main result.
\begin{theorem} \label{MainThm}
Let $G$ be a non-elementary countable group acting properly on a geodesic metric space with contracting elements. Consider    a finite collection of Morse subgroups $H_1,\cdots, H_n$ with infinite index. Then there is an injective $\mathbb{R}$-linear map $\omega:\ell^1 \rightarrow H^2_b(G;\mathbb R)$ such that, each coclass in the image $\omega(\ell^1)$ \rev{has a representative vanishing} on $H_i$ for each $i$  $(1\leq i\leq n)$.

Moreover, the dimension of $H^{2}_b(G, \{H_i\}_{i=1}^n; \mathbb{R})$ as a vector space over $\mathbb{R}$ has the cardinality of the continuum.
\end{theorem}

Here $\ell^1$ denotes the Banach space of summable sequences of real numbers with the norm $\|(x_i)\|=\sum\limits_{i=1}^{\infty}|x_i|$.

\begin{remark}
    \begin{enumerate}
        \item \revise{If one of the $H_i$'s is of finite index in $G$, then Corollary \ref{FiniteIndex2} below implies that $H^{2}_b(G, \{H_i\}_{i=1}^n; \mathbb{R})$=0. Based on this fact, we only consider the second relative bounded cohomology of $G$ with respect to subgroups with infinite index.} Since any finitely generated subgroup of  a free group is Morse, Theorem \ref{MainThm} generalizes the result of Pagliantini-Rolli \cite{PR15} which states that for a free group $F_n(n\ge 2)$ and a finitely generated subgroup $H\le F_n$, the subgroup $H$ has infinite index in $F_n$ if and only if the dimension of the second relative bounded cohomology $H^2_b(F_n,H;\mathbb R)$ as a vector space over $\mathbb R$ is infinite.
        \item In \cite{FPS15, HO13}, the authors extend a non-trivial quasimorphism on a subgroup to the ambient group. Here we take the opposite approach, by constructing a quasimorphism on the ambient group whose restriction to the subgroup is trivial.
    \end{enumerate}
\end{remark}

\revise{To prove Theorem \ref{MainThm}, we need the following result.

\begin{proposition}[Proposition \ref{Prop: Summary}]\label{IntroProp: QT}
    Let $G$ be a non-elementary countable group acting properly on a geodesic metric space with contracting elements. Consider a finite collection of Morse subgroups $H_1,\cdots, H_n$ with infinite index. Then there is a quasi-tree on which $G$ acts and each $H_i$ acts elliptically on it for $1\le i\le n$.
\end{proposition}
}

An interesting corollary of Theorem \ref{MainThm} is as follows. This generalizes a result of Kotschick \cite[Corollary 11]{Kot04}. See Definition \ref{Def: BoundedGeneration} for the definition of bounded generation.
\begin{corollary}[Corollary \ref{MainCor}]
Under the assumption of Theorem \ref{MainThm},  $G$ is not boundedly generated by $\{H_i: 1\le i\le n\}$.
\end{corollary}

When a subgroup is a Morse subgroup of infinite index, its limit set is always a proper subset of the limit set of the ambient group, which provides evidence for the existence of enough relative quasimorphisms. (See \cite{HYZ23, Yan22b} for more details about convergence boundary.) Thus, we can ask the following question.

\begin{question}
    Let $G$ be a non-elementary countable  group acting properly on a geodesic metric space $X$ with convergence boundary. Let $\{H_1,\ldots, H_n\}$ be a finite collection of subgroups with proper limit sets. Then is the dimension of $H^{2}_b(G, \{H_i\}_{i=1}^n; \mathbb{R})$ as a vector space over $\mathbb{R}$  infinite?
\end{question}

Another natural question is that when the subgroup is taken to be a normal subgroup of infinite index, does the conclusion of Theorem \ref{MainThm} still hold? If the normal subgroup has an amenable quotient, then Proposition \ref{Prop: AmenableQuotient} gives a negative answer. If the normal subgroup is normally generated by a higher power of a contracting element, then we have the following result.
\begin{proposition}[Proposition \ref{Prop: NormalClosure}]\label{IntroProp}
    Let $G$ be a non-elementary countable group acting properly on a geodesic metric space with contracting elements. Then for any contracting element $g\in G$, there exists $k=k(g)>0$ such that the dimension of $H^{2}_b(G, \langle\langle g^k\rangle\rangle; \mathbb{R})$ as a vector space over $\mathbb{R}$ has the cardinality of the continuum.
\end{proposition}

\rev{
\begin{remark}
    Proposition \ref{IntroProp} is already known for non-elementary hyperbolic groups. In \cite{Del96}, Delzant showed that for any hyperbolic element $g$ in a non-elementary hyperbolic group $G$, there exists $k\in \N$ such that $G/\llangle g^k\rrangle$ is still hyperbolic. Together with \cite[Theorem 1.1]{EF97} and Proposition \ref{Prop: NormalSubgp} below, one gets the result.
\end{remark}
}
In general, we raise the following question.

\begin{question}\label{IntroQue}
     Let $G$ be a non-elementary countable  group acting properly on a geodesic metric space $X$ with contracting elements. Let $H$ be a normal subgroup of $G$ with non-amenable quotient. Is the dimension of $H^{2}_b(G, H; \mathbb{R})$ as a vector space over $\mathbb{R}$  infinite?
\end{question}

\subsection{Sketch of proof}

There are two main ingredients in the proof of Theorem \ref{MainThm}. The first is \revise{Proposition \ref{IntroProp: QT}, namely} the construction of an appropriate projection complex on which each $H_i$ acts elliptically. This notion is introduced by Bestvina-Bromberg-Fujiwara in \cite{BBF15}. In Section \ref{Sec: MorseSubgpofInfIndex}, we are going to make use of the Morse property of $\{H_i\}$ (cf. Lemma \ref{Lem: UniformShortProjection}) to show that there exists a contracting element $g$ such that every $h\in H_i$ has a uniformly bounded projection to the geodesic segment $[o,go]$. \revise{To achieve this goal, we use some techniques developed by Han-Yang-Zou \cite{HYZ23}.} Then we construct the projection complex $\mathcal P_K(\mathcal F)$ whose vertices are the $G$-translates of the axis $\ax(g)$ of $g$ and show that each $H_i$ acts elliptically on $\mathcal P_K(\mathcal F)$ (cf. Lemma \ref{Lem: EllipticAction}). \revise{The projection complex is shown to be a quasi-tree on which $G$ acts acylindrically (cf. Section \ref{Sec: MainThm}). }

The second ingredient is Proposition \ref{Prop: KeyProp}\revise{, which states that if $G$ acts WPD (a weaker notion than acylindrical action) on a $\delta$-hyperbolic space $X$ and each $H_i$ acts elliptically on $X$, then the dimension of $H^2_b(G, \{H_i\}_{i=1}^n; \mathbb R)$ as a vector space over $\mathbb R$ has the cardinality of the continuum}. To prove this proposition, we need a result of Bestvina-Fujiwara \cite[Proposition 2]{BF02} which constructs an infinite collection of words $\{f_i\}$ in a rank two free subgroup of $G$. As in \cite{Fuj98}, we can produce a corresponding collection of quasimorphisms $\{h_{i}\}$ on $G$, which satisfy some special properties stated in Proposition \ref{Prop: InfiniteQM}. Then Proposition \ref{Prop: KeyProp} follows from Proposition \ref{Prop: InfiniteQM}. The bridge connecting two ingredients is a result of Bestvina-Bromberg-Fujiwara-Sisto \cite[Theorem 5.6]{BBFS19}. As a result, we get Theorem \ref{MainThm}.

As for Proposition \ref{IntroProp}, \revise{where $H$ is assumed to be a normal subgroup generated by a higher power of a contracting element}, we first construct a quasi-tree of spaces $\C(\F)$ for every contracting element which is a blow-up of the projection complex. The space $\C(\F)$ turns out to be a quasi-tree on which $G$ acts acylindrically (cf. Lemma \ref{Lem: AcyActionOnC(F)}). Later, we construct a hyperbolic cone-off over a scaled $\C(\F)$ along $\F$ and obtain a very rotating family (cf. Lemma \ref{Lem: RotFam}). With the help of the theory of rotating families developed by Dahmani-Guirardel-Osin \cite{DGO17}, we can reduce the proof of Proposition \ref{IntroProp} to the proof of Proposition \ref{Prop: WPDQuotient}, which is completed at the end of Section \ref{Sec: ConstructingQM}.


\paragraph{\textbf{Structure of the paper}}
The paper is organized as follows. Section \ref{Sec: Preliminaries} is devoted to recalling some preliminary materials about Gromov-hyperbolic spaces, contracting subsets, projection complexes, (relative) bounded cohomology, and quasimorphisms. In Section \ref{Sec: MorseSubgpofInfIndex}, we \revise{proves Proposition \ref{IntroProp: QT}. Specifically, we} make use of the recent work of Han-Yang-Zou \cite{HYZ23} to construct a projection complex and show that each $H_i$ acts elliptically on it. In Section \ref{Sec: ConstructingQM}, we review the previous work on Epstein-Fujiwara quasimorphisms in \cite{Fuj98} and prove Proposition \ref{Prop: KeyProp}. Then we complete the proof of Theorem \ref{MainThm} in Section \ref{Sec: MainThm}. In Section \ref{sec: RotationFamily}, we first recall some facts about hyperbolic cone-offs and rotation families, and then we prove Proposition \ref{IntroProp} from there. 

\subsection*{Acknowledgments}
We are grateful to Prof. Wenyuan Yang and Prof. Shi Wang for many helpful suggestions on the first draft. Many thanks to Binxue Tao for pointing out Proposition \ref{IntroProp} for us. \revise{We also thank the anonymous referees for their numerous and useful comments.} R. W. is supported by NSFC No.12471065 \& 12326601 and in part by Science and Technology Commission of Shanghai Municipality (No. 22DZ2229014). \revise{Both authors are supported by  National Key R \& D Program of China (SQ2020YFA070059) and NSFC (No. 12131009).}

\section{Preliminaries}\label{Sec: Preliminaries}

We first introduce some fundamental notations and definitions that will be used throughout this paper.

\subsection{\revise{Gromov-hyperbolic spaces}}\label{Subsec: GromovHypSpace}
\revise{We only introduce some necessary knowledge about Gromov-hyperbolic spaces here. For a more detailed introduction, we refer the readers to \cite[Part III. H]{BH99} or \cite[Chapter 11]{DK18}.} Let $(X, d)$ be a geodesic metric space. For $S\subset X$ and $r>0$, we denote by  $N_r(S)$ the \revise{open} $r$-neighborhood of $S$. \revise{For two subsets $S,T\subset X$, we denoted by $d_H(S,T)$ the \textit{Hausdorff distance} between $A$ and $B$, which is defined by $d_H(S,T):= \inf\{r>0: S\subset N_r(T), T\subset N_r(S)\}$.}

For any two points $x,y\in X$, denote by $[x,y]$ a choice of a geodesic segment between $x$ and $y$. A \textit{geodesic triangle} in $X$ consists of three points $x, y, z\in X$ and three geodesic segments $[x, y], [y, z]$, and $ [z, x]$.

A geodesic metric space $(X,d)$ is called \textit{(Gromov) $\delta$-hyperbolic} for a constant $\delta \geq 0$ if every geodesic triangle in $X$ is \textit{$\delta$-thin}: each of its sides is contained in the $\delta$-neighborhood of the union of the other two sides. \revise{See Figure \ref{Fig: ThinTri} for an illustration.}

\begin{figure}[ht]
  \centering
  \includegraphics[width=10cm]{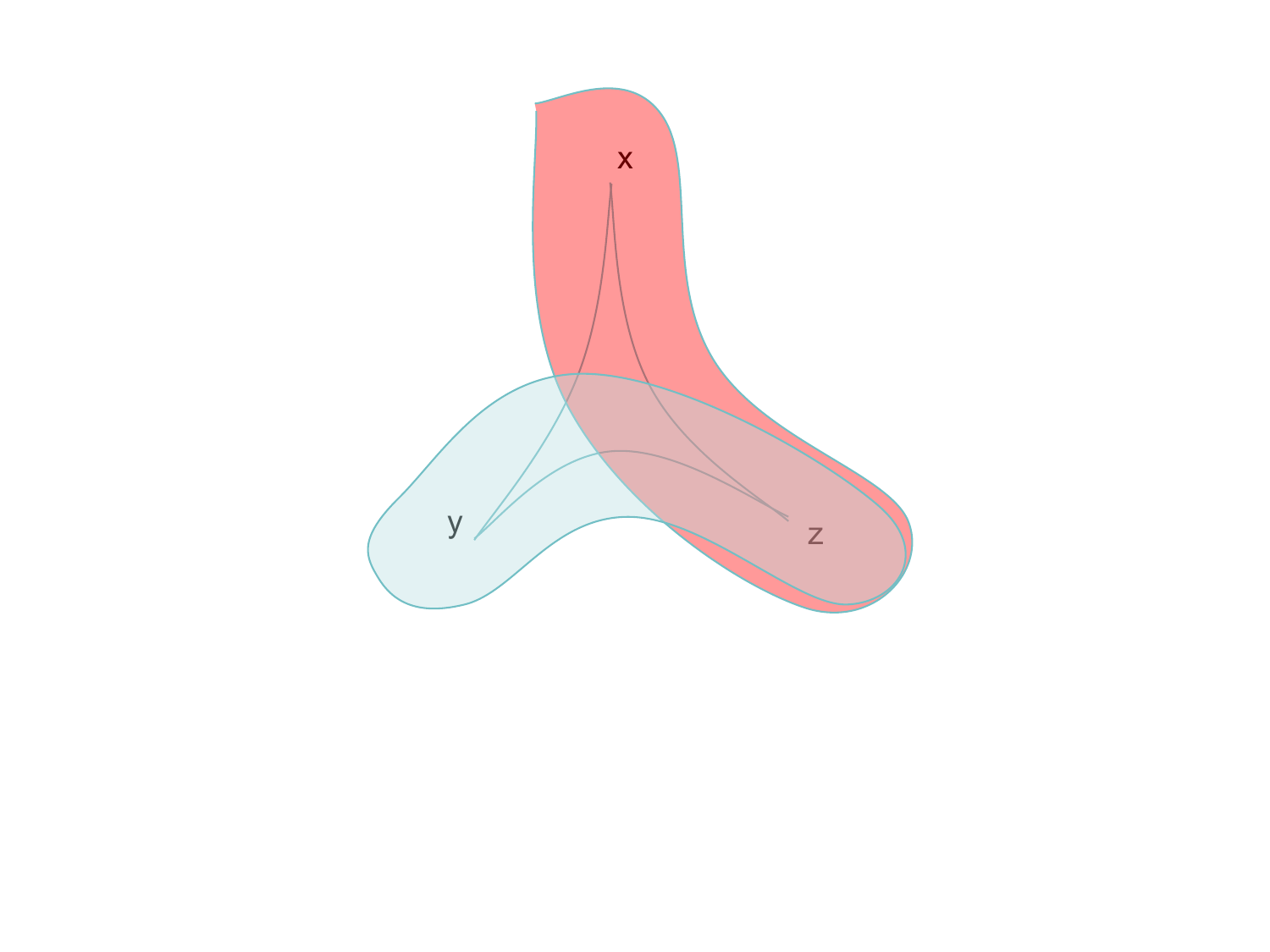}\\
  \caption{$[x,y]\subset N_{\delta}([x,z]\cup [y,z])$}\label{Fig: ThinTri}
\end{figure}

A finitely generated group is called \revise{\textit{Gromov-hyperbolic}}, if its Cayley graph with respect to some finite generating set is a $\delta$-hyperbolic metric space for some $\delta\geq 0$.  

Let $(X_1, d_1)$ and $(X_2, d_2)$ be two metric spaces. A (not necessarily continuous) map $f: X_1\rightarrow X_2$ is called a \textit{$(\lambda, \epsilon)$-quasi-isometric embedding} if there exist constants $\lambda\geq 1$ and $\epsilon\geq 0$ such that for all $x, y \in X_1$ we have
$$
\frac{1}{\lambda}d_1(x, y)-\epsilon \leq d_2(f(x), f(y)) \leq\lambda d_1(x, y)+\epsilon.
$$                                                                                               If, in addition, every point of $X_2$ lies in the $\epsilon$-neighborhood of the image of $f$, then $f$ is called a \textit{$(\lambda, \epsilon)$-quasi-isometry}. When such a map exists, the two spaces $X_1$ and $X_2$ are said to be \textit{quasi-isometric.}

A \textit{$(\lambda, \epsilon)$-quasi-geodesic} in a metric space $X$ is the image of a
 $(\lambda, \epsilon)$-quasi-isometric embedding $c: I \rightarrow X$, where $I$ is an interval (possibly \revise{bounded or unbounded}). For simplicity, a $(\lambda,\lambda)$-quasi-geodesic will be referred to as a \textit{$\lambda$-quasi-geodesic}.

 \revise{The following result is the well-known Morse Lemma in Geometric Group Theory. We refer the readers to \cite[Chapter III.H, Theorem 1.7]{BH99} or \cite[Theorem 11.72]{DK18} for a proof.}

\begin{lemma}[Morse Lemma]\label{Lem: MorseLemma}
    \revise{For all $\delta\ge 0, \lambda\ge 1, \epsilon\ge 0$, there exists a constant $L=L(\lambda,\epsilon,\delta)$ such that any two $(\lambda, \epsilon)$-quasi-geodesics in a $\delta$-hyperbolic space with the same endpoints are contained in an $L$-neighborhood of each other}.
\end{lemma}

\paragraph{\textbf{Isometries on \revise{Gromov-hyperbolic spaces}}} \revise{Let $(X,d)$ be a Gromov-hyperbolic space. Two geodesic rays in $X$ are said to be \textit{asymptotic} if the Hausdorff distance between them is finite. Being asymptotic is an equivalence relation on geodesic rays. The \textit{Gromov boundary} of $X$, denoted by $\partial X$, is defined to be the set of equivalence classes of geodesic rays in $X$. We refer the readers to \cite{BH99, DK18} for a detailed discussion about Gromov boundary.} By Gromov \cite{Gro87}, the isometries of a hyperbolic space $X$ can be subdivided into three classes. A nontrivial element $g\in \mathrm{Isom}(X)$ is called \textit{elliptic} if some $\langle g\rangle$-orbit is bounded. Otherwise, it is called \textit{loxodromic} (resp. \textit{parabolic}) if it has exactly two fixed points (resp. one fixed point)  in the Gromov boundary of $X$. If $g$ is a loxodromic element, any $\langle g\rangle$-invariant quasi-geodesic between the two fixed points will be referred to as  a \textit{quasi-axis} for $g$, denoted by $L_g$.

For an isometry $g$ on a hyperbolic space $(X,d)$, we define the \textit{stable translation length} $\Vert g\Vert$ of $g$ as the limit $$\Vert g\Vert:=\lim_{n\to\infty}\frac{d(x,g^nx)}{n}$$ for some (or any) $x\in X$. A well-known fact is that an isometry $g$ is loxodromic if and only if $\Vert g\Vert>0$ \revise{\cite[Chapter 10, Proposition 6.3]{CDP90}}.

\begin{lemma}\label{Lem: StableLength}
    Let $g$ be an isometry on a metric space $(X,d)$. Then for any $n>0$ and $x\in X$, one has $d(x,g^nx)\ge n\Vert g\Vert$.
\end{lemma}
\begin{proof}
    For any $m,n>0$, it follows from $d(x,g^{mn}x)\le m\cdot d(x,g^nx)$ that $\frac{d(x,g^{mn}x)}{mn}\le \frac{d(x,g^nx)}{n}$. By letting $m\to \infty$, one gets that $\Vert g\Vert\le \frac{d(x,g^nx)}{n}$.
\end{proof}


\subsection{Contracting \revise{subsets}}
Let $(X,d)$ be a geodesic metric space. For any two subsets $A,B\subset X$, define $d(A,B):=\inf_{x\in A, y\in B}d(x,y)$. For a given \revise{closed} subset $A\subset X$, define the \textit{closest point projection} $\pi_A: X\to \revise{2^A}$ as follows:
    \begin{itemize}
        \item for any $x\in X$, $\pi_A(x):=\{y\in A: d(x,y)=d(x,A)\}.$
        \item for any subset $B\subset X$, $\pi_A(B):=\cup_{x\in B}\pi_A(x)$.
    \end{itemize}
\revise{We use the notation $\diam(A)$ to denote the diameter of a subset $A$ in $X$.}
\begin{definition}[Contracting Subset]\label{DEF: Contracting Subset}
Let $(X,d)$ be a geodesic metric space. A subset $Y\subseteq X$ is called $C$-\textit{contracting} for $C\geq 0$ if for any geodesic (segment) $\alpha$ in $X$ with $d(\alpha,Y)\geq C$, we have $\diam(\mathrm{\pi}_Y(\alpha))\leq C$.
$Y$ is called a \textit{contracting subset} if there exists $C\geq 0$ such that $Y$ is $C$-contracting, and $C$ is called a \revise{\it{contraction constant}} of $Y$.

\begin{figure}[ht]
  \centering
  \includegraphics[width=10cm]{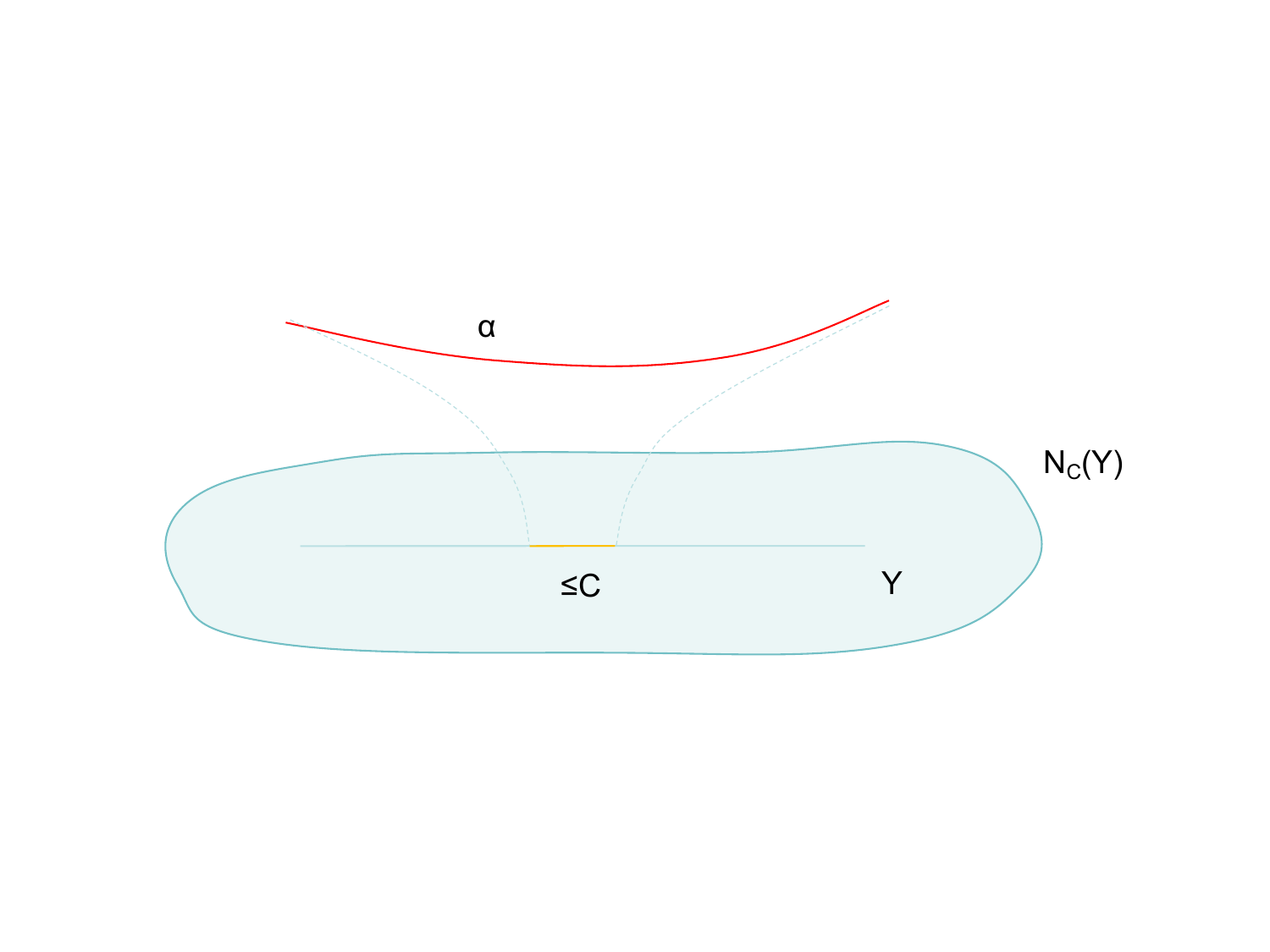}\\
  \caption{$Y$ is $C$-contracting}
\end{figure}

\revise{Let $G$ be a finitely generated group. A subgroup $H\le G$ is called \textit{contracting} if it is a contracting subset in the Cayley graph $\G(G, S)$ of $G$ with respect to some finite generating set $S$.}
\end{definition}
\begin{example}\label{EX: Contracting Subset}
The following are well-known examples of contracting subsets \revise{and contracting subgroups}.
\begin{enumerate}
    \item Bounded sets in a metric space.
    \item Quasi-geodesics and quasi-convex subsets in Gromov-hyperbolic spaces. \cite{Ghy90}
    \item Fully quasi-convex subgroups, and maximal parabolic subgroups in particular, in relatively hyperbolic groups. \cite[Proposition 8.2.4]{Ger15}
    \item The subgroup generated by a hyperbolic element in groups with non-trivial Floyd boundary. \cite[Section 7]{Yan14}
    \item Contracting segments and axes of rank-$1$ elements in $\mathrm{CAT}(0)$-spaces in the sense of Bestvina and Fujiwara. \cite[Corollary 3.4]{BF09}
    \item The axis of any pseudo-Anosov element in the Teichm\"uller space equipped with Teichm\"uller distance by \revise{a theorem of} Minsky. \cite{Min96}
\end{enumerate}
\end{example}

It has been proven in \cite[Corollary 3.4]{BF09} that the definition of a contracting subset is equivalent to the following one considered by  Minsky \cite{Min96}:

A subset $Y\subseteq X$ is contracting if and only if there exists $C'\geq 0$ such that any open metric ball $B$ with $B\cap Y=\emptyset$ satisfies $\diam(\mathrm{\pi}_Y(B))\leq C'$.

Despite this equivalence, we will always rely on Definition \ref{DEF: Contracting Subset} for a contracting subset.
\begin{definition}[Contracting System]
    Let $(X,d)$ be a geodesic metric space. \revise{A set} $\mathbb{X}=\{X_i: X_i\subset X\}_{i\in \mathbb N}$ is called a \textit{contracting system} with a contraction constant $C$ if each $X_i$ is a $C$-contracting subset in $X$.
\end{definition}

\begin{definition}[Bounded Intersection]\label{Def: BoundedIntersection}
Two subsets $Y,Z\subseteq X$ have $\mathcal{R}$-\textit{bounded intersection} for a function $\mathcal{R}:[0,+\infty)\to[0,+\infty)$ if $\diam(N_r(Y)\cap N_r(Z))\leq \mathcal{R}(r)$, for all $ r\geq 0$. A contracting system $\mathbb X$ has $\mathcal{R}$-\textit{bounded intersection} if any two elements in $\mathbb X$ have $\mathcal{R}$-\textit{bounded intersection}.
\end{definition}

\paragraph{\textbf{Admissible \revise{paths}}}
\revise{For a finite rectifiable path $p$ in a metric space $X$, we denote by $|p|$ the length of $p$ and by $p_-, p_+$ the initial and terminal points of $p$ respectively. }
\begin{definition}[Admissible Path]\label{DEF: Admissible Path}
Let $(X,d)$ be a geodesic metric space and $\mathbb{X}$ be a contracting system in $X$.  Let $D,\tau\geq 0$ and $\mathcal{R}:[0,+\infty)\to[0,+\infty)$ be a function, which is called the \textit{bounded intersection gauge}. A path $\gamma$ is called a $(D,\tau)$-\textit{admissible path}  with respect to $\mathbb{X}$, if the path $\gamma$ consists of a (finite, infinite, or bi-infinite) \revise{concatenation} of consecutive geodesic segments $\gamma=\cdots q_ip_iq_{i+1}p_{i+1}\cdots$, satisfying the following ``Long Local'' and ``Bounded Projection'' properties:

\begin{itemize}
\item[\textbf{(LL1)}]\label{LL1} For each $p_i$ there exists $X_i\in\mathbb{X}$ such that the two endpoints of $p_i$ lie in $X_i$, and $|p_i|>D$ unless $p_i$ is the first or last geodesic segment in $\gamma$.
\item[\textbf{(BP)}]\label{BP} For each $X_i$ we have $\diam(\pi_{X_i}\{(p_i)_+,(p_{i+1})_-\})\leq \tau$, and $\diam(\pi_{X_i}\{(p_{i-1})_+, (p_i)_-\})\leq \tau$. Here $(p_{i+1})_-=\gamma_+$ if $p_{i+1}$ does not exist, and $(p_{i-1})_+=\gamma_-$ if $p_{i-1}$ does not exist.
\item[\textbf{(LL2)}]\label{LL2} Either $X_i\neq X_{i+1}$ and $X_i$ and $X_{i+1}$ have $\mathcal{R}$-bounded intersection, or $\revise{|q_i|}>D$.
\end{itemize}
\end{definition}

\begin{remark}\label{Rmk: Saturation}
    \revise{See Figure \ref{Fig: AdmiPath} for an illustration of an admissible path.} The collection of $X_i\in \mathbb X$ indexed \revise{as} in (LL1), denoted by $\mathbb X(\gamma)$, will be referred to as \textit{contracting subsets} for $\gamma$. \revise{The union of all $X_i\in \X(\gamma)$ is called the \textit{saturation} of $\gamma$.}
\end{remark}

\begin{figure}[ht]
  \centering
  \includegraphics[width=12cm]{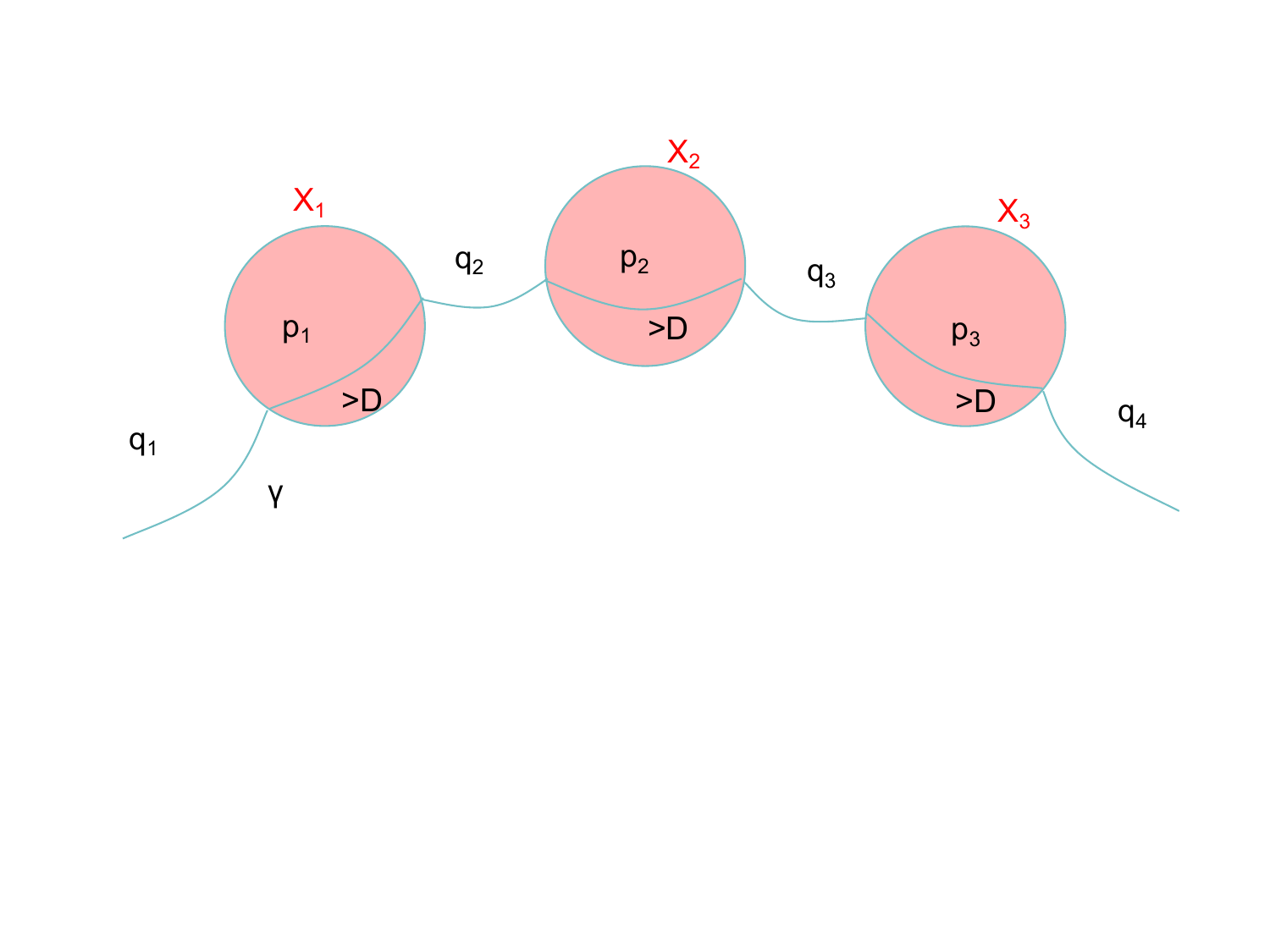}\\
  \caption{$\gamma=q_1p_1q_2p_2q_3p_3q_4$ is a $(D,\tau)$-admissible path}\label{Fig: AdmiPath}
\end{figure}

In the following definitions, a sequence of points $x_i$ in a path $\alpha$ is called \textit{linearly ordered} if \revise{$x_{i+1}$ lies in the subpath of $\alpha$ from $x_i$ to $\alpha_+$} for each $i$.

\begin{definition}[Fellow Travel]\label{DEF: Fellow Travel}
Let $\gamma=p_0q_1p_1\cdots q_np_n$ be a $(D,\tau)$-admissible path, and $\alpha$ be a path such that $\alpha_-=\gamma_-$, $\alpha_+=\gamma_+$. Given $\epsilon>0$, the path $\alpha$ $\epsilon$-\textit{fellow travels} $\gamma$ if there exists a sequence of linearly ordered points $z_0,w_0,z_1,w_1,\cdots,z_n,w_n$ on $\alpha$ such that $d(z_i,(p_i)_-)\leq \epsilon$, $d(w_i,(p_i)_+)\leq \epsilon$.
\end{definition}
\revise{See Figure \ref{Fig: FellowTravel} for an illustration of $\epsilon$-fellow travel property.}

\begin{figure}[ht]
  \centering
  \includegraphics[width=12cm]{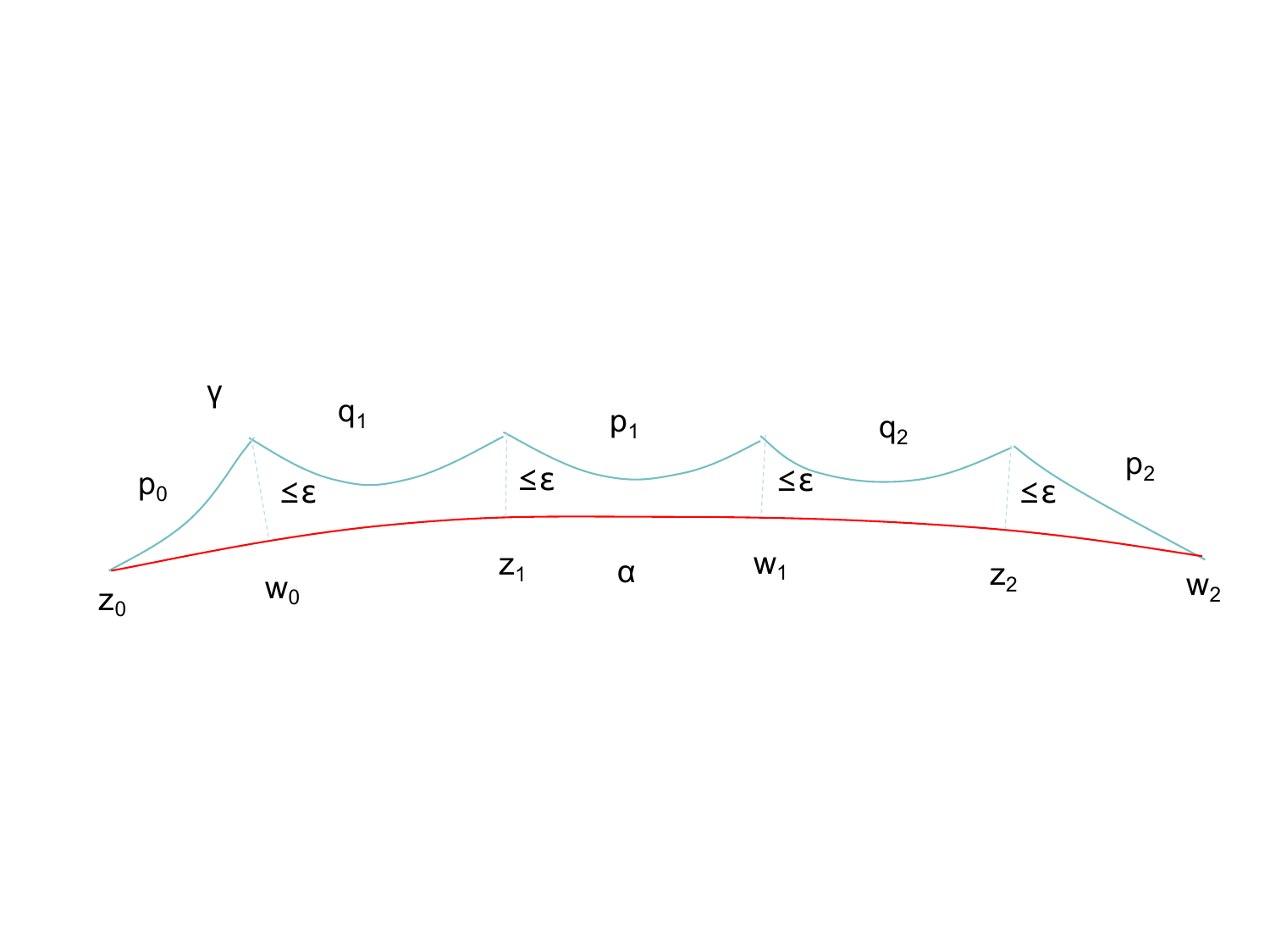}\\
  \caption{$\gamma=p_0q_1p_1q_2p_2$ is a $(D,\tau)$-admissible path and $\alpha$ $\epsilon$-fellow travels $\gamma$}\label{Fig: FellowTravel}
\end{figure}

\begin{proposition}[{\revise{\cite[Proposition 2.7]{Yan19}}} ]\label{PROP: Fellow Travel}
For any $\tau>0$\revise{, and any function $\mathcal R:[0,+\infty)\to [0,+\infty)$}, there exist $D,\lambda, \epsilon>0$ depending on $\tau,\revise{\mathcal R}$ such that any $(D,\tau)$-admissible path is a $\lambda$-quasi-geodesic which \revise{is $\epsilon$-fellow traveled by} any geodesic with the same endpoints.
\end{proposition}

\paragraph{\textbf{Group \revise{actions} with contracting elements}}

Let $G$ be a group acting isometrically on a geodesic metric space $(X,d)$ with a base point $o\in X$.

\begin{definition}[Contracting Element]\label{DEF: Contracting Element}
An element $h\in G$ is called a \textit{contracting element} if $\left\langle h\right\rangle\cdot o$ is a contracting subset in $X$ and the map $\mathbb{Z}\to X$, $n\mapsto h^no$ is a quasi-isometric embedding.
\end{definition}

\revise{A group action $G\curvearrowright X$ is called \textit{(metrically) proper} if for any $D>0$, the set $\{g\in G: d(o,go)\le D\}$ is finite.} From now on, we assume that $G$ acts properly on $(X,d)$ with a contracting element.

\begin{definition}[Weakly-independent]\label{Def: weakly-independence}
Suppose $g,h\in G$ are two contracting elements. \revise{They} are called \textit{weakly-independent} if $\left\langle g\right\rangle\cdot o$ and $\left\langle h\right\rangle\cdot o$ have $\mathcal{R}$-bounded intersection for some $\mathcal{R}:[0,+\infty)\to[0,+\infty)$.
\end{definition}

By \cite[Lemma 2.11]{Yan19}, each contracting element $g$ is contained in a maximal elementary subgroup $E(g)$ defined as follows:
$$E(g) = \{h \in G : \exists r > 0, h\langle g\rangle o\subset  N_r(\langle g\rangle o) \text{ and } \langle g\rangle o\subset  N_r(h\langle g\rangle o)\},$$
and the index $[E(g) : \langle g\rangle]$ is finite. \revise{Roughly speaking, the subgroup $E(g)$ is the set of elements that do not move the orbit of $g$ too much.} Hence, if $G$ is not virtually cyclic, then there are at least two weakly-independent contracting elements in $G$. \revise{For example, let $g\in G$ be a contracting element by assumption. Since $G$ is not virtually cyclic, $E(g)$ is a proper subgroup of $G$. By selecting an element $f\in G\setminus E(g)$, one has that $g$ and $fgf^{-1}$ are weakly-independent \cite[Lemma 7.13]{WXY24}.} Actually, by \cite[Lemma 2.30]{WXY24}, there exist infinitely many pairwise weakly-independent elements in $G$.

For a contracting element $g$, the subset denoted by $\ax(g) := E(g)o$ is called an \textit{extended-defined quasi-geodesic}, which is also called the \textit{axis} of $g$ (depending on $o$). Compared with the quasi-axis $L_g$ of a loxodromic element $g$ which can be an arbitrary $\langle g\rangle$-invariant quasi-geodesic, we use the term ``extended-defined'' here to emphasize their difference.

The following result, proved in \cite[Proposition 2.9 and Lemma 2.14]{Yan19}, gives a procedure to
construct \revise{infinitely many} contracting elements.

\begin{lemma}\revise{[Extension Lemma]}\label{Lem: ExtensionLemma}
    Suppose that a non-elementary group $G$ acts properly on $(X, d)$ with a contracting element. Then there exist a set $F \subset G$ of three contracting elements and $D,\tau,\lambda > 0$ with the following property.

    For any $g \in G$, there exists $f \in F$ such that the bi-infinite concatenated path $\gamma=\cup_{n\in \mathbb Z}(gf)^n[o,gfo]$ is a $(D,\tau)$-admissible path with contracting subsets $\{(gf)^ng\ax(f): n\in \mathbb Z\}$.  In particular, $\gamma$ is a $C$-contracting $\lambda$-quasi-geodesic, where $C$ depends on $d(o,go)$. Hence $gf$ is a contracting element.
\end{lemma}

\subsection{Projection \revise{complexes}}\label{subsec: ProjectionComplex}
\revise{In the assumption of our main result, i.e. Theorem \ref{MainThm}, the geodesic metric space $X$ on which the group $G$ acts with contracting elements is not hyperbolic. To apply the method proposed by Bestvina-Fujiwara \cite{BF02}, we need to construct a suitable hyperbolic space on which $G$ acts. This construction makes use of the projection complex techniques developed by Bestvina-Bromberg-Fujiwara \cite{BBF15}, which will be introduced in this subsection. }
\begin{definition}[Projection axioms]\label{Def: ProjectionAxioms}
    Let $\mathcal F$ be a collection of metric spaces equipped with (set-valued) projection maps
    $$\{\pi_U: \mathcal F\setminus\{U\}\to \revise{2^U}\}_{U\in \mathcal F}.$$

Denote $d_U(V,W):=\revise{\diam(\pi_U(V)\cup \pi_U(W))}$ for $V\neq U\neq W\in \mathcal F$. The pair $(\mathcal F,\{U\}_{U\in \mathcal F})$ satisfies projection axioms for a constant $\kappa>0$ if
\begin{itemize}
    \item[(1)] $\revise{\diam(\pi_U(V))}\le \kappa$ when $U\neq V$.
    \item[(2)] if $U, V,W$ are distinct and $d_V (U,W) > \kappa$ then $d_U (V,W) \le \kappa$.
    \item[(3)] the set $\{U \in \mathcal F: d_U (V,W) > \kappa\}$ is finite for $V \neq W$.
\end{itemize}
\end{definition}
\rev{We caution the readers that the projection maps in the above definition are just abstract maps, which may not be the closest point projections defined previously.}
By definition, the following triangle inequality holds for any $U,V,Y,W\in \mathcal F$:
\begin{equation}\label{ProjectionTriangleInequa}
    d_Y (V,W) \le d_Y (V, U) + d_Y (U,W).
\end{equation}
It is well known that the projection axioms hold for a contracting system with bounded intersection (cf. \cite[Appendix]{Yan14}). From now on, we assume that $G$ acts properly on $X$ with contracting elements. Fix a base point $o\in X$.

\begin{lemma}\cite[Lemma 2.18]{Yan22b}\label{Lem: one element}
    \revise{Let} $g\in G$ be a contracting element. Then the collection $\mathcal F = \{f\ax(g) : f \in G\}$ with closest point projections $\pi_U (V )$ satisfies the projection axioms with a constant $\kappa = \kappa(\mathcal F) > 0$.
\end{lemma}

In \cite[Definition 3.1]{BBF15}, a modified version of $d_U$ is introduced such that it is symmetric and agrees with the
original $d_U$ up to an additive amount $2\kappa$. Thus, the axioms (1)-(3) still hold for $3\kappa$, and the triangle
inequality in (\ref{ProjectionTriangleInequa}) holds up to a uniform error. In what follows, we actually need to work with this modified $d_U$ to define the projection complex, but for the sake of simplicity, we stick \revise{to} the above definition of $d_U$.

We consider the interval-like set for $K > 0$ and $V,W \in \mathcal F$ as follows
$$\mathcal F_K(V,W) := \{U \in \mathcal F : d_U (V,W) > K\}.$$
Denote $\mathcal F_K[V,W] := \mathcal F_K(V,W)\cup \{V,W\}$. It possesses a total order described in the next lemma. \rev{Let $\F$ and $\kappa$ be given by Lemma \ref{Lem: one element}.}
\begin{lemma}\cite[Theorem 3.3.G]{BBF15}\label{TotalOrder}
     There exist constants $D = D(\kappa), K = K(\kappa) > 0$ such that the set $\mathcal F_K[V,W]$ admits a total order ``$<$'' with least element $V$ and greatest element $W$, such that given $U_0, U_1, U_2 \in \mathcal F_K[V,W]$, if $U_0 < U_1 < U_2$, then
$$d_{U_1}(V,W) - D \le d_{U_1}(U_0, U_2) \le d_{U_1}(V,W),$$
and
$$d_{U_0}(U_1, U_2) \le D \text{ and } d_{U_2}(U_0, U_1) \le D.$$
\end{lemma}
We now give the definition of a projection complex.
\begin{definition}\label{Def: ProjectionCpx}
    The projection complex $\mathcal P_K(\mathcal F)$ for $K$ satisfying Lemma \ref{TotalOrder} is a graph with the vertex set consisting of the elements in $\mathcal F$. Two vertices $U$ and $V$ are connected if $\mathcal F_K(U, V ) = \emptyset$. We equip $\mathcal P_K(\mathcal F)$ with a length metric $d_{\mathcal P}$ induced by assigning unit length to each edge.
\end{definition}
The projection complex $\mathcal P_K(\mathcal F)$ is connected\revise{, since} by \cite[Proposition 3.7]{BBF15}, the interval set $\mathcal F_K[U, V ]$ gives a connected path between $U$ and $V$ in $\mathcal P_K(\mathcal F)$: the consecutive elements directed by the total order are adjacent in $\mathcal P_K(\mathcal F)$. \revise{The path $\F_K[U,V]\cup \{U,V\}=\{U< U_1< \cdots < U_k< V\}$ is called the \textit{standard path} from $U$ to $V$. See Figure \ref{Fig: LiftStanPath} for an illustration of a standard path.} By \cite[Corollary 3.7]{BBFS19}, standard paths are (2, 1)-quasi-geodesics \revise{in $\P_K(\F)$}. \revise{A \textit{quasi-tree} is a geodesic metric space quasi-isometric to a tree.} The structural result about the projection complex is the following.
\begin{lemma}\cite[Corollary 3.25]{BBF15}\label{Lem: PCQuasi-tree}
    For $K \gg 0$ as in Lemma \ref{TotalOrder}, the projection complex $\mathcal P_K(\mathcal F)$ is a quasi-tree, on which $G$ acts non-elementarily and co-boundedly.
\end{lemma}

\begin{figure}[ht]
  \centering
  \includegraphics[width=10cm]{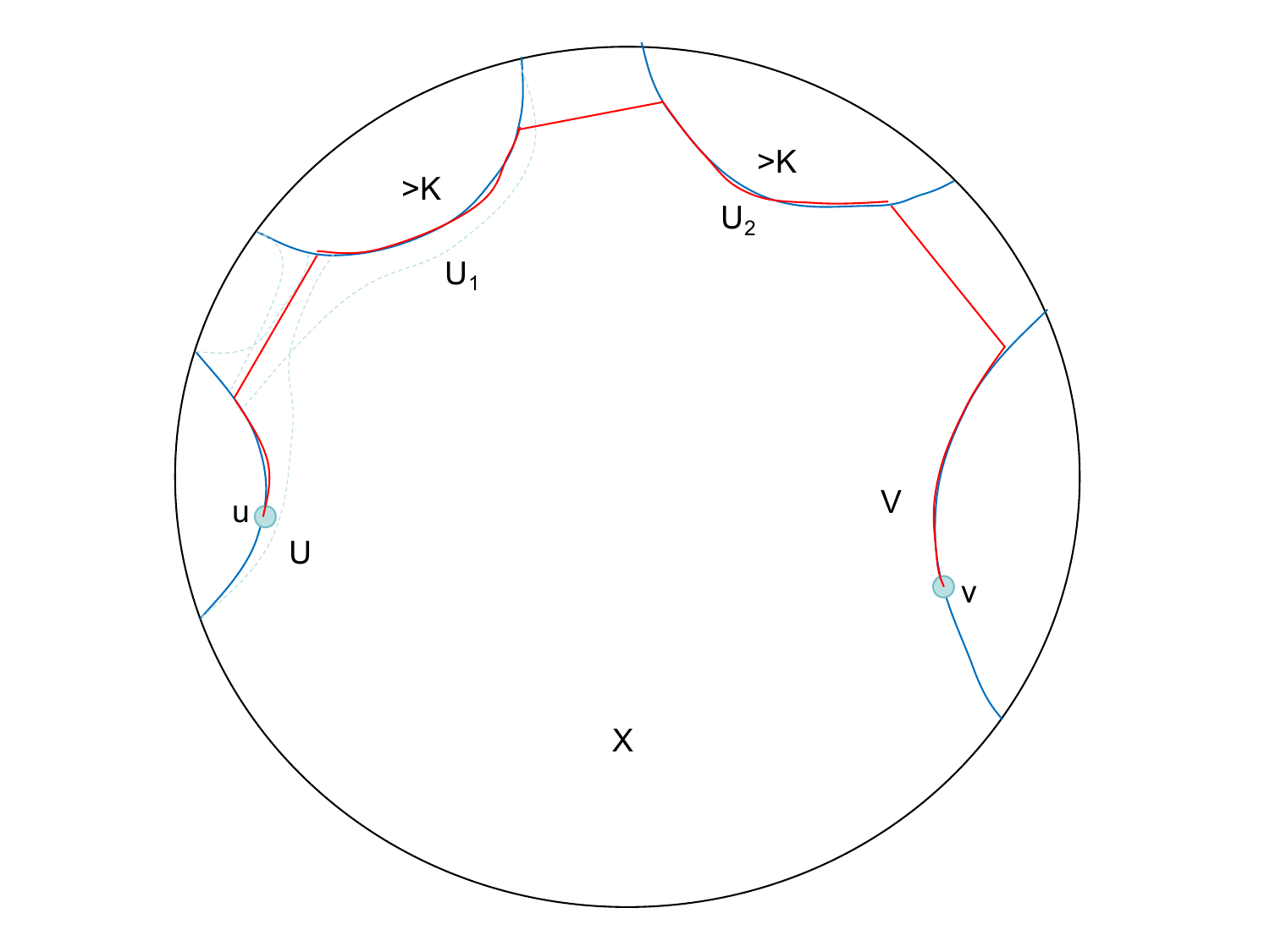}\\
  \caption{Each blue line is a translate of $\ax(g)$; $\F_K[U,V]\cup \{U,V\}=\{U<U_1<U_2<V\}$ is a standard path in $\P_K(\F)$; the dashed line represents the closest point projection between $U$ and $U_1$; the red line is a lifted standard path from $u\in U$ to $v\in V$ in $X$}\label{Fig: LiftStanPath}
\end{figure}

For any two points $u \in U$ and $v \in V$, we often need to \textit{lift} a standard path \revise{$\F_K[U,V]\cup\{U,V\}=\{U< U_1< \cdots < U_k< V\}$ in $\P_K(\F)$} to \revise{a path from $u$ to $v$ in} $X$. \revise{The \textit{lifted standard path} in $X$} is a piecewise geodesic path (admissible path) \revise{from $u$ to $v$} as concatenation of the normal paths between two consecutive vertices \revise{which are $G$-translates of $\ax(g)$ in $X$} and geodesics contained in vertices. \revise{See Figure \ref{Fig: LiftStanPath} for an illustration of a lifted standard path.} This is explained by the following lemma proved in \cite[Lemma 4.5]{HLY20}.
\begin{lemma}\label{Lem: AdmissiblePath}
    For any $K > 0$, there exist a constant $D = D(K, \kappa) \ge 0$ with $D \to \infty$ as $K\to \infty$ and a uniform constant $B = B(\kappa) > 0$ with the following property. For any two points $u \in U, v \in V$ there exists an $(D, B)$-admissible path $\gamma$ in $X$ from $u$ to $v$ with saturation $\mathcal F_K[U, V ]$.
\end{lemma}

In practice, we always assume that $K$ \revise{is sufficiently large such that by taking $\tau=B$ and $\mathcal R$ as the bounded intersection gauge of $\mathcal F$, the constant $D$ in the above lemma} satisfies Proposition \ref{PROP: Fellow Travel}, \revise{and then} the path $\gamma$ shall be a quasi-geodesic.


\subsection{(Relative) bounded cohomology}\label{Subsec: RelBoundedCohomology}

This subsection is devoted to the basic definitions and results on the bounded cohomology of a group and the relative bounded cohomology of a pair of groups. We refer the reader to \cite{Bro94} for full details on ordinary group cohomology theory, to \cite{Fri17, Gro82, Iva85}  on bounded cohomology of  groups, and to \cite{Gro82, Par03} for the relative case.

Let $G$ be a group. \revise{For a coefficient ring $R$, the \textit{bar complex} $C_*(G; R)$} is the complex generated in dimension $n$ by $n$-tuples $(g_1, \ldots, g_n)$ with $g_i \in G$ and with boundary map $\partial$ defined by the formula $$\partial (g_1, \ldots, g_n) = (g_2, \ldots, g_n)+\sum_{i=1}^{n-1}(-1)^i(g_1, \ldots , g_ig_{i+1}, \ldots, g_n)+(-1)^n(g_1, \ldots, g_{n-1}).$$
We let $C^*(G; R)$ denote the dual cochain complex $\mathrm{Hom}(C_*(G), R)$, and let $d$ denote the adjoint of $\partial$. The homology groups of $C^*(G; R)$ are called the \textit{group cohomology of $G$ with coefficients in $R$}, and are denoted $H^*(G; R)$.

In this paper, we take $R\revise{=}\mathbb R$. A cochain $\alpha \in C^n(G; \mathbb R)$ is \textit{bounded} if $$\sup |\alpha(g_1, \ldots , g_n)| < \infty$$ where the supremum is taken over all \revise{$n$-tuples $(g_1,\ldots,g_n)$ with $g_i\in G$}. The set of all bounded cochains forms a subcomplex $C^*_b(G;\mathbb R)$ of $C^*(G;\mathbb R)$, and its homology is the so-called \textit{bounded cohomology} $H_b^*(G;\mathbb R)$.

\paragraph{\textbf{Amenable groups}}

Recall that a \textit{mean} on $G$ is a linear functional on $C^1_b(G;\mathbb R)$ which maps the constant function $\phi(g) = 1$ to 1, and maps non-negative functions to non-negative numbers.

\begin{definition}\label{Def: Amenable}
    A group $G$ is \textit{amenable} if there is a $G$-invariant mean $\pi : C^1_b(G;\mathbb R) \to \mathbb R$ where $G$ acts on $C^1_b(G;\mathbb R)$ by $$g\cdot \phi(h)=\phi(g^{-1}h)$$ for all $g, h \in G$ and $\phi \in C^1_b(G;\mathbb R)$.
\end{definition}

Examples of amenable groups are finite groups, solvable groups, and Grigorchuk's groups of intermediate growth.

We list three important facts here for future use. Some of them are mentioned in the Introduction. See \cite[\textsection 2.4.2]{Cal09} or \cite{Bou95} for details.
\begin{lemma}\label{Lem: Amenable}
    \begin{enumerate}
        \item $H^{1}_b(G; \mathbb{R})=0$ for any group $G$.
        \item $H_b^*(G;\mathbb R)$ vanishes identically when $G$ is an amenable group.
        \item Let $H\to G\to K\to 1$ be a (right) exact sequence of groups. Then the induced sequence on second bounded cohomology is (left) exact: $$0\to H^2_b(K; \mathbb R)\to H^2_b(G; \mathbb R)\to H^2_b(H; \mathbb R).$$
    \end{enumerate}
\end{lemma}

Unless otherwise stated, \textit{a pair of groups} $(G,H)$ in this paper means that $G$ is a group and $H$ is a subgroup of $G$. Now we come to the definition of relative bounded cohomology of the pair $(G, H)$. The kernel of the obvious restriction map $C^{*}_b(G,\mathbb{R}) \rightarrow C^{*}_b(H, \mathbb{R})$ is denoted by $C^{*}_b(G, H; \mathbb{R})$, and we have the short exact sequence of complexes 

\begin{center}
$0 \rightarrow C^{*}_b(G, H; \mathbb{R}) \rightarrow C^{*}_b(G; \mathbb{R}) \rightarrow C^{*}_b(H; \mathbb{R}) \rightarrow 0$
\end{center}

\noindent which induces a long exact sequence in cohomology

\begin{equation}\label{Equ: LongExactSeq}
\cdots \rightarrow H^{n-1}_b(H; \mathbb{R}) \rightarrow H^n(C^{*}_b(G,H; \mathbb{R})) \rightarrow H^{n}_b(G; \mathbb{R}) \rightarrow H^{n}_b(H; \mathbb{R}) \rightarrow \cdots .
\end{equation}

The module $H^n(C^{*}_b(G,H; \mathbb{R}))$ is the $n$-th bounded cohomology of the pair $(G, H)$. We denote it as $H^n_b(G,H; \mathbb{R})$. \revise{Let $\{H_i\}_{i=1}^m$ be a finite collection of subgroups in $G$. We denote by $C^{\ast}_b(G,\{H_i\}_{i=1}^m; \R)$ the kernel of the multiple restriction map $C^{*}_b(G,\mathbb{R}) \rightarrow \bigoplus_{i=1}^mC^{*}_b(H_i, \mathbb{R})$. In the same way as above, we define $H^n_b(G,\{H_i\}_{i=1}^m; \mathbb{R})$ as the $n$-th bounded cohomology $H^n(C^{*}_b(G,\{H_i\}_{i=1}^m; \mathbb{R}))$.} There are no substantial differences from the relative cohomology in ordinary cohomology group theory.

\revise{
\begin{proposition}\label{Prop: NormalSubgp}
    Let $(G,H)$ be a pair of groups with $H\lhd G$. Then $H_b^2(G/H; \R)\cong H^2_b(G, H; \mathbb R)$.
\end{proposition}
\begin{proof}
    Recall that Lemma \ref{Lem: Amenable} (3) gives us a left-exact sequence $$0\to H^2_b(G/H; \mathbb R)\to H^2_b(G; \mathbb R)\to H^2_b(H; \mathbb R).$$

    From the definition of relative bounded cohomology (\ref{Equ: LongExactSeq}), we also have the following exact sequence $$H^{1}_b(H; \mathbb{R}) \rightarrow H_b^2(G,H; \mathbb{R}) \rightarrow H^{2}_b(G; \mathbb{R}) \rightarrow H^{2}_b(H; \mathbb{R}).$$ As any group has a trivial first bounded cohomology, we can conclude that
    \begin{equation}\label{equ: Isom}
        H^2_b(G/H; \mathbb R)\cong \ker(H^2_b(G; \mathbb R)\to H^2_b(H; \mathbb R))\cong H^2_b(G,H; \mathbb R).
    \end{equation}
\end{proof}

We remark that the second isomorphism in (\ref{equ: Isom}) always holds as long as $(G,H)$ is a pair of groups.
}
\paragraph{\textbf{Finite-index subgroups}}

Let $(G,H)$ be a pair of groups. The inclusion map from $H$ to $G$ induces the \textit{restriction map} which is denoted by $res:H^n(G, \mathbb R)  \rightarrow H^n(H, \mathbb R)$.  A subgroup $H$ of a topological group $G$ is called \textit{cocompact} (or uniform) if the quotient space $G/\overline H$ is compact, where $\overline H$ denotes the closure of $H$ in $G$. For cocompact subgroups such that the quotient admits a finite invariant measure, the restriction map from the cohomology group of \revise {the ambient group} to the cohomology group of the \revise{subgroup} is injective. This notably encompasses the case of uniform lattices (of Lie groups) and finite index subgroups (of any group). The standard argument uses the existence of a \textit{transfer map} which is the left inverse to the restriction map (or up to a constant multiple depending on different definitions of the transfer map). The transfer map is obtained by integration (or finite sum for finite index subgroups), so it is crucial for the subgroup to be cocompact. See \cite[\textsection 8.6]{Mon01} for a detailed discussion of the transfer map.

\begin{lemma}\cite[Propositon 8.6.2]{Mon01}\label{FiniteIndex1}
Let $(G, H)$ be a pair of groups. If $H$ has finite index in $G$, then the natural restriction map $H^{2}_b(G; \mathbb{R}) \rightarrow H^{2}_b(H; \mathbb{R})$ is isometrically injective.
\end{lemma}

\revise{Combining with the remark following Proposition \ref{Prop: NormalSubgp}, we have immediately that}

\begin{corollary} \label{FiniteIndex2}
Let $(G, H)$ be a pair of groups. If $H$ has finite index in $G$, then $H^{2}_b(G, H; \mathbb{R})=0$.

\end{corollary}



The corollary above gives us a better understanding of the relations between the index of a subgroup and the relative bounded cohomology group. In fact, Corollary \ref{FiniteIndex2} is the main motivation for the authors of \cite{PR15} and us to consider the relative bounded cohomology of infinite index subgroups, aiming to find some non-trivial results. \revise{By combining Lemma \ref{Lem: Amenable} (2) and Proposition \ref{Prop: NormalSubgp}, one gets another simple example as follows.}

\begin{proposition}\label{Prop: AmenableQuotient}
    Let $(G,H)$ be a pair of groups. If $H$ is a normal subgroup with amenable quotient\revise{, then} $H^2_b(G,H;\mathbb R)=0$.
\end{proposition}

\subsection{Quasimorphisms}\label{subsec: QM}
In this subsection, we recall the notion of quasimorphism on a group $G$. \revise{Typically, one proves that $H^2_b(G;\R)$ is infinite-dimensional by demonstrating the existence of infinitely many linearly independent quasimorphisms on $G$.}

\begin{definition}
A map $\phi: G \rightarrow \mathbb{R}$ is called a \textit{quasimorphism} if there exists a constant $C>0$ such that
$$
|\phi(gh)-\phi(g)-\phi(h)|<C, ~~\forall g, h \in G
.$$
The \textit{defect} of a quasimorphism $\phi$ is defined to be:
\begin{center}
 $\Delta(\phi) :=\sup\limits_{g,  h \in G}|\phi(gh)-\phi(g)-\phi(h)|$.
\end{center}
The defect of a quasimorphism measures how far it is from a genuine homomorphism from the group to $\mathbb{R}$. A quasimorphism $\phi$ is called \textit{trivial} if there exists a bounded map $b: G \rightarrow \mathbb{R}$ and a group homomorphism $\rho :G \rightarrow \mathbb{R}$ such that $\phi=b+\rho$.
\end{definition}

We denote by $\rm{QM}(G)$ the $\mathbb{R}$-vector space of quasimorphisms on $G$ and by $\rm{QM}_0(G)$ the subspace of $\rm{QM}(G)$ consisting of all trivial quasimorphisms on $G$, which is exactly $C^1_b(G; \mathbb{R})\bigoplus \rm{Hom}(G, \mathbb{R})$.

For a quasimorphism $\phi$ on $G$, define the 1-coboundary of $\phi$  by $d^1 \phi(g, h)=\phi(g)+\phi(h)-\phi(gh)$. Thus, $d^1 \phi$ is a bounded 2-cocycle. We denote by $\omega_{\phi}:=[d^1 \phi]_b$ the corresponding bounded cohomology class. Then we have a linear map

\begin{center}
$\rm{QM}(G)\rightarrow H_b^2(G; \mathbb{R})$

$\phi\mapsto \omega_{\phi}$
\end{center}

\noindent such that the following sequence

\begin{center}
$\rm{QM}(G)\rightarrow H_b^2(G; \mathbb{R}) \rightarrow H^2(G; \mathbb{R})$

\end{center}
\noindent is exact \revise{\cite[Theorem 2.50]{Cal09}}.

There is another important special kind of quasimorphism called a \textit{homogeneous quasimorphism}.

\begin{definition}
A quasimorphism $\phi: G \rightarrow \mathbb{R}$ is homogeneous if $\phi(g^n)=n\cdot \phi(g)$ for every $g\in G$ and every $n\in \mathbb{Z}$.
\end{definition}

\revise{A \textit{class function} on $G$ is a function that takes the same value on each conjugacy class.
\begin{lemma}
    A homogeneous quasimorphism is a class function.
\end{lemma}
\begin{proof}
    Let $\phi$ be a homogeneous quasimorphism on $G$. For any two group elements $f,g\in G$, it follows from the definition of homogeneous quasimorphisms that
    \begin{align*}
        |\phi(g)-\phi(fgf^{-1})|& =|\phi(g)+\phi(fg^{-1}f^{-1})|=\frac{|\phi(g^n)+\phi(fg^{-n}f^{-1})|}{n}\\ &\le \frac{|\phi(g^n)+\phi(g^{-n})|+|\phi(f)|+|\phi(f^{-1})|+2\Delta(\phi)}{n}= \frac{2|\phi(f)|+2\Delta(\phi)}{n}.
    \end{align*}
    By letting $n\to \infty$, one has that $\phi(g)=\phi(fgf^{-1})$. Since $f,g$ are arbitrary, we complete the proof.
\end{proof}
}

\begin{remark}\cite[Lemma 2.21, Corollary 2.59]{Cal09}\label{Rmk: HomogeneousQG}
Let $\phi$ be a quasimorphism on $G$. Then there exists a unique homogeneous quasimorphism $\bar{\phi}$ which stays at finite distance from $\phi$. In fact, the corresponding homogeneous quasimorphism to $\phi$ is given by
$$\forall g\in G, \quad \bar{\phi}(g):=\lim \limits_{n\to\infty}\frac{\phi(g^n)}{n}.$$
\noindent This limit exists because the coarse subadditive inequality ${\phi(g^{n+m})}\le \phi(g^n)+\phi(g^m)+\Delta(\phi)$ holds. Moreover, the defect $\Delta(\bar \phi)$ is related to $\Delta(\phi)$ by $\Delta(\bar \phi)\le 2\Delta(\phi)$.
\end{remark}

We denote the subspace of $\rm{QM}(G)$ consisting of all homogeneous quasimorphisms on $G$ as $\rm{QM}_h(G)$.

\section{Morse subgroups of infinite index}\label{Sec: MorseSubgpofInfIndex}

In this section, we assume that $G$ is a non-elementary group acting properly on a geodesic metric space $(X,d)$ with contracting elements. Fix a base point $o\in X$. \revise{We first define what is a Morse subgroup.}

\revise{
\begin{definition}[Morse Property]\label{Def: MorseSubgp}
    A subset $A \subset X$ is \textit{$\eta$-Morse} for a function $\eta: \mathbb R\to \mathbb R$ if every $\lambda$-quasi-geodesic with endpoints in $A$ is contained in the $\eta(\lambda)$-neighborhood of $A$. The function $\eta$ is called a \textit{Morse gauge} of $A$. A subgroup $H \le G$ is \textit{Morse} if the subset $H\cdot o$ is $\eta$-Morse for some function $\eta: \mathbb R\to \mathbb R$.
\end{definition}

\begin{remark}\label{Rmk: FiniteHausDist}

        \rev{A $C$-contracting set is $\eta_0$-Morse for some function $\eta_0 : \mathbb R_{\ge 0} \to \mathbb R_{\ge 0}$ depending only on $C$ \cite[Lemma 2.8(1)]{Sis18}.}

\end{remark}}

\revise{From now on, we assume that $\{H_i\}_{1\le i\le n}$ is a finite collection of Morse subgroups with infinite index in $G$. The purpose of this section is to prove Proposition \ref{IntroProp: QT}, which is restated as
\begin{proposition}\label{Prop: Summary}
    There exists a quasi-tree on which $G$ acts co-boundedly and each $H_i$ acts elliptically for $1\le i\le n$.
\end{proposition}

\noindent\textbf{Proof ideas of Proposition \ref{Prop: Summary}:}
For a contracting element $g\in G$, Subsection \ref{subsec: ProjectionComplex} shows that there exists a  projection complex $\P_K(\F)$ where $\F=\{f\ax(g): f\in G\}$ such that $\P_K(\F)$ is a quasi-tree on which $G$ acts co-boundedly. Hence, we only need to find an appropriate contracting element $g\in G$ such that each $H_i$ acts elliptically on $\P_K(\F)$ for $1\le i\le n$. By analyzing the algebraic and geometric properties of $\{H_i: 1\le i\le n\}$, we can obtain a contracting element $g$ such that the projection of $[o,ho]$ onto $\ax(g)$ is uniformly bounded for any $h\in \cup_{1\le i\le n}H_i$. Finally, we will show that such a contracting element $g$ meets the requirements.

In order to characterize the magnitude of the (closest point) projection to an axis of a contracting element, we introduce a definition of barriers.
}

\begin{definition}\cite[Definition 4.1]{Yan19}\label{Def: Barriers}
    Let $\epsilon\ge 0$ and $f\in G$. We say a geodesic $\alpha\subset X$ contains an \textit{$(\epsilon,f)$-barrier} if there exists an element $t\in G$ such that the following holds:
    $$d(to,\alpha)\le \epsilon, \quad d(tfo,\alpha)\le \epsilon.$$
    Otherwise, $\alpha$ is called \textit{$(\epsilon,f)$-barrier-free}. An element $g\in G$ is called \textit{$(\epsilon,f)$-barrier-free} if \revise{some choice of geodesic from $o$ to $go$} is $(\epsilon,f)$-barrier-free.
\end{definition}

\begin{figure}[ht]
  \centering
  \begin{tikzpicture}
\draw [black] (-4,0)--(4,0);
\filldraw [red] (-3,1) circle (1pt) node[black, above] {$to$};
\filldraw [red] (3,1) circle (1pt) node[black, above] {$tfo$};
\draw [black] (-3,1) .. controls (0,0.2) .. (3,1);
\draw [dashed, black] (-3,1)--(-3,0);
\draw [dashed, black] (3,1)--(3,0);
\draw (-3,0.4) node[black,left] {$\leq \epsilon$};
\draw (3,0.4) node[black,right] {$\leq \epsilon$};
\draw (0,-0.5) node[black,right] {$\alpha$};
\end{tikzpicture}
  \caption{$\alpha$ contains an $(\epsilon,f)$-barrier}
\end{figure}

Recall that \revise{the Extension Lemma, i.e. Lemma \ref{Lem: ExtensionLemma}} gives a set $F\subset G$ consisting of three contracting elements. \revise{The following result of Han-Yang-Zou shows that for any contracting element $g=g_0f$ obtained by Extension Lemma, any geodesic segment with a large projection to $\ax(g)$ contains an $(\epsilon,g_0)$-barrier where $\epsilon$ depends only on $F$.}
\begin{lemma}\cite[Lemma 2.9]{HYZ23}\label{Lem: UniformEpsilon}
    For any $g_0\in G$, let $g=g_0f$ be the contracting element given by Lemma \ref{Lem: ExtensionLemma}. Then there exist $\epsilon=\epsilon(F), \tau=\tau(F,g)$ such that a geodesic segment $\alpha$ with $\diam(\pi_{\ax(g)}(\alpha))>\tau$ contains an  $(\epsilon,g_0)$-barrier.
\end{lemma}

\begin{remark}\label{Rmk: UniformEpsilon}
    Actually, one can strengthen the conclusion of Lemma \ref{Lem: UniformEpsilon} as follows: if a geodesic segment $\alpha$ satisfies $\diam(\pi_{b\ax(g)}(\alpha))>\tau$ for some $b\in G$, then $\alpha$ contains an $(\epsilon,g_0)$-barrier. The reason is that from the definition of barriers (cf. Definition \ref{Def: Barriers}), it is equivalent to say that $\alpha$ or $b^{-1}\alpha$ contains an $(\epsilon,g_0)$-barrier. Thus, by substituting $\alpha$ with $b^{-1}\alpha$, one can apply Lemma \ref{Lem: UniformEpsilon} to obtain the stronger conclusion.
\end{remark}

\revise{To get a contracting element such that the projection of each $[o,ho]$ to $\ax(g)$ is uniformly bounded, we need to find an element $g_0$ such that $h$ is $(\epsilon,g_0)$-barrier-free for $h\in \cup_{1\le i\le n}H_i$.}

\begin{lemma}\cite[Lemma 4.1]{Neu54b}\label{Lem: FiniteUnion}
    Let the group $G$ be the union of finitely many cosets of subgroups $C_1,\ldots, C_n$. Then the index of (at least) one of these subgroups in $G$ does not exceed $n$.
\end{lemma}

\begin{lemma}\label{Lem: g1g2}
For any $\epsilon\ge 0$, there exists an element $g_0\in G$ such that any $h\in H_i$ with $1\le i \le n$ is $(\epsilon,g_0)$-barrier-free.
\end{lemma}

\begin{proof}
    At first, we claim that the set $G\setminus S(\cup_{1\le i\le n} H_i)S$ is infinite for any finite subset $S\subset G$. Suppose not, by enlarging $S$ by a finite set, we can assume that $G=S(\cup_{1\le i\le n} H_i)S=\cup_{1\le i\le n} (SH_iS)=\cup_{s\in S}\cup_{1\le i\le n}(sH_is^{-1})(s\cdot S)$ where each $s\cdot S$ is a finite set. As each conjugate of $H_i$ is still an infinite-index subgroup of $G$, the above decomposition implies that $G$ can be written as  a finite union of cosets of infinite-index subgroups. This contradicts Lemma \ref{Lem: FiniteUnion}. The claim thus follows.

    \revise{Since each $H_i$ is Morse for $1\le i\le n$, it follows from Definition \ref{Def: MorseSubgp} that there exists a function $\eta: \R\to \R$ such that all orbits $H_i\cdot o(1\le i\le n)$ are $\eta$-Morse.} Let $M=\eta(1)$. Denote $S:=\{g\in G: d(o,go)\leq \epsilon+M\}$. As $G$ acts properly on $X$, the set $S$ is finite.

    \revise{According to the above claim, the set $G\setminus S(\cup_{1\le i\le n} H_i)S$ is infinite. We are going to show that every element $g_0\in G\setminus S(\cup_{1\le i\le n} H_i)S$ satisfies the requirement. Suppose to the contrary that there exist some $i\in \{1,\dots, n\}$ and some element $h\in H_i$ such that $h$ is not $(\epsilon, g_0)$-barrier-free. By definition of barriers, for any geodesic segment $[o,ho]$, there exist an element $b\in G$ and two points $x,y\in [o,ho]$ such that $d(bo,x), d(bg_0o,y)\le \epsilon$. \rev{See the following figure for an illustration.}

    \begin{figure}[ht]
  \centering
  \begin{tikzpicture}
\draw [black] (-4,0)--(4,0);
\filldraw [red] (-3,1) circle (1pt) node[black, above] {$bo$};
\filldraw [red] (3,1) circle (1pt) node[black, above] {$bg_0o$};
\filldraw [red] (4,0) circle (1pt) node[black, below] {$ho$};
\filldraw [red] (-4,0) circle (1pt) node[black, below] {$o$};
\draw [black] (-3,1) .. controls (0,0.2) .. (3,1);
\draw [dashed, black] (-3,1)--(-3,0);
\draw [dashed, black] (3,1)--(3,0);
\filldraw [red] (3,0) circle (1pt) node[black, below] {$y$};
\filldraw [red] (-3,0) circle (1pt) node[black, below] {$x$};
\draw (-3,0.4) node[black,left] {$\leq \epsilon$};
\draw (3,0.4) node[black,right] {$\leq \epsilon$};
\end{tikzpicture}
  \caption{$[o,ho]$ contains an $(\epsilon,g_0)$-barrier}
\end{figure}

    Since the orbit $H_i\cdot o$ is $\eta$-Morse in $X$, by definition of Morse property, the geodesic segment $[o,ho]$ is contained in $N_M(H_i o)$. Since $x,y\in [o,ho]$, there exist $h_1,h_2\in H_i$ such that $d(x,h_1o),d(y,h_2o)\le M$. Combining two inequalities together, one gets that $d(bo,h_1o),d(bg_0o,h_2o)\le \epsilon+M$. By construction of $S$, we have $b^{-1}h_1, h^{-1}_2bg_0 \in S$. Therefore, $$g_0 \in b^{-1}h_2S=b^{-1}h_1\cdot h_1^{-1}h_2S\subset SH_iS,$$ which contradicts with the choice of $g_0$. The conclusion then follows.}
\end{proof}

\revise{By combining Remark \ref{Rmk: UniformEpsilon} and Lemma \ref{Lem: g1g2}, we obtain that}

\begin{lemma}\label{Lem: UniformShortProjection}
There exist a contracting element $g\in G$ and  $\tau>0$ such that for any $b\in G$ and $h\in H_i$ with $1\le i \le n$,
$$\mathrm{diam}(\pi_{b\ax(g)}([o,ho]))\leq \tau.$$
\end{lemma}
\begin{proof}
    \revise{We caution the readers that the constant $\epsilon:=\epsilon(F)$ given by  Lemma \ref{Lem: UniformEpsilon} only depends on a pre-fixed set $F$. Then} we choose $g_0$ by Lemma \ref{Lem: g1g2} such that each $h\in H_i$ with $1\le i \le n$ is $(\epsilon,g_0)$-barrier-free. Set $g=g_0f$ given by Lemma \ref{Lem: ExtensionLemma} for some $f\in F$. Then  $\tau=\tau(F,g)$ in Lemma \ref{Lem: UniformEpsilon} is the desired uniform projection constant. Otherwise, one gets a contradiction to Remark \ref{Rmk: UniformEpsilon}.
\end{proof}

\revise{For the contracting element given by Lemma \ref{Lem: UniformShortProjection}, we next follow the process in Subsection \ref{subsec: ProjectionComplex} to construct a projection complex. Finally, we show that the projection complex is the quasi-tree which satisfies all requirements of Proposition \ref{Prop: Summary}. }

\paragraph{\textbf{Group actions on projection complex}} 

Let $g$ be a fixed contracting element given by Lemma \ref{Lem: UniformShortProjection}. Lemma \ref{Lem: one element} shows that $\mathcal F=\{f\ax(g): f\in G\}$ with shortest projection maps satisfies the projection axioms with constants $\kappa=\kappa(\mathcal F)>0$. Hence, one can construct a projection complex $\mathcal P_K(\mathcal F)$ (cf. Definition \ref{Def: ProjectionCpx}) for $K\gg 0$. As a result of Lemma \ref{Lem: PCQuasi-tree}, $\mathcal P_K(\mathcal F)$ is a quasi-tree, on which $G$ acts non-elementarily and co-boundedly. Set $K\ge \tau+2\kappa+2\epsilon$ where $\tau$ is given by Lemma \ref{Lem: UniformShortProjection}, $\kappa$ is given by Lemma \ref{Lem: one element}, and $\epsilon$ is the fellow travel constant (cf. Proposition \ref{PROP: Fellow Travel}) with respect to a $(D,B)$-admissible path given by Lemma \ref{Lem: AdmissiblePath}.

\begin{lemma}\label{Lem: EllipticAction}
    For each $i\in \{1,\ldots, n\}$, $H_i$ acts elliptically on $\mathcal P_K(\mathcal F)$.
\end{lemma}
\begin{proof}
    Recall that the vertex set of $\mathcal P_K(\mathcal F)$ is $\mathcal F=\{f\ax(g): f\in G\}$, and two vertices $U,V$ are connected by an edge if the set $\mathcal F_K[U,V]= \{Z \in \mathcal F : d_Z (U,V) > K\}$ is empty. Let $U\in \mathcal P_K(\mathcal F)$ be the point representing $\ax(g)$ \revise{which by definition is the orbit $E(g)\cdot o$}. It suffices for us to show that $d_{\mathcal P}(U,hU)\le 1$ for any $h\in H_i$ with $1\le i\le n$.

    Suppose not; then there exists at least one element in $\mathcal F_K[U,hU]$. Let $V\in \mathcal F_K[U,hU]$. It follows from the definition of $\mathcal F_K[U,hU]$ that $\diam(\pi_V(U)\cup \pi_V(hU))>K$. Note that $o\in U$ and thus $ho\in hU$. Lemma \ref{Lem: AdmissiblePath} shows that there exists a $(D,B)$-admissible path in $X$ from $o$ to $ho$ with saturation $\mathcal F_K[U,hU]$. It follows from the $\epsilon$-fellow travel property of an admissible path that
    $$\diam(\pi_V([o,ho]))\ge \diam(\pi_V(U)\cup \pi_V(hU))-2\kappa-2\epsilon>K-2\kappa-2\epsilon\ge \tau.$$
    As $V$ represents a $G$-translate of $\ax(g)$, one gets  a contradiction to Lemma \ref{Lem: UniformShortProjection}. The conclusion then follows.
\end{proof}

\revise{As a corollary of Lemma \ref{Lem: EllipticAction}, we get Proposition \ref{Prop: Summary}.
\begin{remark}
    With Lemma \ref{Lem: EllipticAction}, we know that the quasi-tree in Proposition \ref{Prop: Summary} is actually a projection complex $\mathcal P_K(\mathcal F)$. Since each vertex in $\mathcal P_K(\mathcal F)$ represents a translate of $\ax(g)$, any conjugate of $E(g)$ fixes a point in $\P_K(\F)$. Hence, the action $G\curvearrowright \mathcal P_K(\mathcal F)$ is not proper.  Furthermore, it is generally difficult to obtain subgroups acting elliptically on $\mathcal P_K(\mathcal F)$ other than the vertex stabilizer. However, Proposition \ref{Prop: Summary} provides such a family of subgroups, i.e. Morse subgroups of infinite index.
\end{remark}
}


\section{Constructing Quasimorphisms}\label{Sec: ConstructingQM}
In this section, we assume that a non-elementary countable group $G$ acts WPD (cf. Definition \ref{Def: WPD}) on a $\delta$-hyperbolic space $X$ and a  finite collection of subgroups $\{H_i: 1\le i\le n\}$ of $G$ acts elliptically on $X$. Our goal is the following result.

\begin{proposition}\label{Prop: KeyProp}
    There is an injective $\mathbb{R}$-linear map $\omega:\ell^1 \rightarrow H^2_b(G;\mathbb R)$ such that each coclass in the image $\omega(\ell^1)$ has a representative vanishing on $H_i$ for each $1\leq i\leq n$.

    Moreover, the
    dimension of $H^2_b(G, \{H_i\}_{i=1}^n; \mathbb R)$ as a vector space over $\mathbb R$ has the cardinality of the continuum.
\end{proposition}

\begin{definition}\cite{BF02}\label{Def: WPD}
    We say that the action of $G$ on a hyperbolic space $X$ satisfies \textit{WPD} if
    \begin{itemize}
        \item $G$ is not virtually cyclic,
        \item $G$ contains at least one element that acts on $X$ as a loxodromic isometry, and
        \item for every loxodromic element $g\in G$, every $x\in X$, and every $C>0$, there exists $N>0$ such that the set $$\{h\in G: d(x,hx)\le C, d(g^Nx, hg^Nx)\le C\}$$ is finite.
    \end{itemize}
\end{definition}

In \cite{BF02}, Bestvina-Fujiwara  showed that the dimension of $\rm{QM}(G)/(H^1(G;\mathbb R)\oplus C_b^1(G;\mathbb R))$ \revise{as a $\R$-vector space has the cardinality of the continuum under the assumption of WPD actions}. This \revise{implies} the absolute version of Proposition \ref{Prop: KeyProp}.
\revise{However, the relative version of second bounded cohomology of a group acting WPD on a hyperbolic space has never been studied before, which is where the value of our Proposition \ref{Prop: KeyProp} lies. Regarding the proof of Proposition \ref{Prop: KeyProp}, we generally follow the proof idea in \cite{BF02}, but we will utilize some new techniques (eg. barriers) developed in this paper.}

\subsection{Epstein-Fujiwara quasimorphisms on groups acting on hyperbolic spaces.}\label{Subsec: EF QG}

At first, let us recall some basic material about Epstein-Fujiwara quasimorphisms introduced in \cite{Fuj98}.

Let $\alpha$ be a finite path in $X$. We denote the length of $\alpha$ by $|\alpha|$. We use the action of $g\in G$ on $X$ to define a path $g\cdot \alpha$ which is the $g$-translation of the path $\alpha$. We say that $g\cdot \alpha $ is a \textit{copy} of $\alpha$. Let $w$ be a finite oriented path, and let $w^{-1}$ be the inverse path. We assume that $|w|\ge 2$ and define
\begin{equation*}
 |\alpha|_w:=\{ \text{the maximal number of copies of $w$ in $\alpha$ without overlapping \revise{(except at the vertices)}} \}. \end{equation*}

Suppose that $x,y \in X$ and that $W$ is a number with $0<W<|w|$. Recall that $[x,y]$ denotes some choice of a geodesic from $x$ to $y$. We define

\begin{equation}\label{Definition of c_w}
    c_{w,W}([x,y])=d(x,y)-\inf\limits_{\alpha}\{|\alpha|-W|\alpha|_w\}
\end{equation}
\noindent where $\alpha$ ranges over all the paths from $x$ to $y$.  \revise{It follows from the definition that $c_{w,W}([x,y])$ does not depend on the choice of a geodesic $[x,y]$.
\begin{remark}\label{Rmk: RealPathNotExist}
    By choosing $\alpha$ to be a choice of geodesic $[x,y]$, one gets that $$c_{w,W}([x,y])\revise{\ge d(x,y)-(|[x,y]|-W|[x,y]|_w)=W|[x,y]|_w}\ge 0.$$ Moreover, if $c_{w,W}([x,y])=0$, then the above inequality implies that $|[x,y]|_w=0$ and the geodesic $[x,y]$ realizes the infimum in (\ref{Definition of c_w}). However,  if $c_{w,W}([x,y])>0$, then the realizing path (i.e. a path realizing the infimum in (\ref{Definition of c_w})) may not exist.
\end{remark}
}

\begin{lemma}\cite[Lemma 3.3]{Fuj98} \label{5.1}
Suppose that a path $\alpha$ realizes the infimum above. Then $\alpha$ is a $\left(\frac{|w|}{|w|-W}, \frac{2W|w|}{|w|-W}\right)$-quasi-geodesic.
\end{lemma}

\revise{
Since a realizing path does not always exist, we need another notion which is close to realizing paths in practice.
A path $\beta$ between $x$ and $y$ is called an \textit{almost realizing path} of $c_{w,W}([x,y])$ if it satisfies that
\begin{equation}\label{Equ: AlmostReal}
    |\beta|-W|\beta|_w \le \min\{\inf_{\alpha}\{|\alpha|-W|\alpha|_w\}+W, \quad (d(x,y)+\inf_{\alpha}\{|\alpha|-W|\alpha|_w\})/2\}
\end{equation}
where $\alpha$ ranges over all the paths from $x$ to $y$. In other words, an almost realizing path $\beta$ of $c_{w,W}([x,y])$ satisfies that $$d(x,y)-(|\beta|-W|\beta|_w) \ge \max \{c_{w,W}([x,y])-W, \quad c_{w,W}([x,y])/2\}.$$
We remark that the requirement $d(x,y)-(|\beta|-W|\beta|_w) \ge c_{w,W}([x,y])-W$ guarantees $\beta$ is  a uniform quasi-geodesic (cf. Lemma \ref{Lem: AlmostRealizing} below) and the other requirement $d(x,y)-(|\beta|-W|\beta|_w) \ge c_{w,W}([x,y])/2$ is used to obtain $|\beta|_w>0$ when $c_{w,W}([x,y])>0$.
By Definition and Remark \ref{Rmk: RealPathNotExist}, an almost realizing path of $c_{w,W}([x,y])$ always exists. Analogous to Lemma \ref{5.1}, we have that
\begin{lemma}\label{Lem: AlmostRealizing}
    Let $\beta$ be an almost realizing path of $c_{w,W}([x,y])$. Then $\beta$ is a $\left(\frac{|w|}{|w|-W}, \frac{3W|w|}{|w|-W}\right)$-quasi-geodesic.
\end{lemma}
\begin{proof}
    Let $\beta: [0,|\beta|]\to X$ be an arc-length parametrization of $\beta$. Let $0\le t< s\le |\beta|$ and set $\beta'=\beta|_{[t,s]}$. Note that $|\beta'|=s-t$. Let $\gamma$ be a geodesic from $\beta(t)$ to $\beta(s)$.
    \begin{claim}
        $|\beta'|-W(|\beta'|_w+3)\le |\gamma|-W|\gamma|_w.$
    \end{claim}
    \begin{proof}[Proof of Claim]
        Suppose to the contrary that $|\beta'|-W(|\beta'|_w+3)> |\gamma|-W|\gamma|_w$. Since $\beta$ is an almost realizing path, $|\beta|-W|\beta|_w \le \inf_{\alpha}\{|\alpha|-W|\alpha|_w\}+W$. By setting $\gamma'=\beta|_{[0,t]}\cup \gamma \cup \beta|_{[s,|\beta|]}$, one has that
        \begin{align*}
            |\gamma'|-W|\gamma'|_w& \le (t+|\gamma|+|\beta|-s)-W(\mid \beta|_{[0,t]}\mid_w+|\gamma|_w+\mid\beta|_{[s,|\beta|]}\mid_w)\\
            &< |\beta|-W(\mid \beta|_{[0,t]}\mid_w+|\beta'|_w+\mid\beta|_{[s,|\beta|]}\mid_w+3)\le |\beta|-W(|\beta|_w+1)\le \inf_{\alpha}\{|\alpha|-W|\alpha|_w\}.
        \end{align*}
        This is impossible since $\gamma'$ is also a path from $x$ to $y$.
    \end{proof}
    Clearly $|\beta'|_w\le |\beta'|/|w|$. Therefore,
    \begin{align*}
        d(\beta(t),\beta(s))=|\gamma|\ge |\gamma|-W|\gamma|_w\ge |\beta'|-W|\beta'|_w-3W\ge |\beta'|-\frac{W}{|w|}|\beta'|-3W=\frac{|w|-W}{|w|}|\beta'|-3W.
    \end{align*}
    The proof is complete.
\end{proof}
}

It is clear from the definition that $c_{w,W}([x, y])=c_{w^{-1},W}([y, x])$. We then define  $$h_{w,W}([x,y])=c_{w,W}([x,y])-c_{w^{-1},W}([x,y]).$$

Take $o\in X$ as a base point. We define functions $c_{w,W}$ and $h_{w,W}: G \rightarrow \mathbb{R}$ by
$$c_{w,W}(g):=c_{w,W}([o, g\cdot o]).$$
$$h_{w,W}(g):=h_{w,W}([o, g\cdot o]).$$

\begin{lemma}\cite[Proposition 3.10]{Fuj98}\label{Lem: UniformDefect}
    The map $h_{w,W}: G\to \mathbb R$ is a quasimorphism. Moreover, the defect $\Delta (h_{w,W}) \leq 12L_0+6W+48\delta$  is uniformly bounded where $L_0=L(\frac{|w|}{|w|-W}, \frac{2W|w|}{|w|-W},\delta)$ is given by Lemma \ref{Lem: MorseLemma}.
\end{lemma}

\subsection{Constructing infinitely many words}
From now on, we borrow some definitions and terminologies from \cite{BF02}. For a base point $o\in X$ and a loxodromic element $g$, \revise{we denote by $L_g=\cup_{i\in \Z}g^{i}[o,go]$ a quasi-axis of $g$. Define $\ax(g):=\cup_{k\ge 1}L_{g^k}$. By Morse Lemma, $\ax(g)$ is still a quasi-axis of $g$. The} quasi-axis $\ax(g)$ of $g$ is oriented by the requirement that $g$ acts as a positive translation. We call this orientation the \textit{$g$-orientation} of the quasi-axis. Of course, the $g^{-1}$-orientation is the opposite of the $g$-orientation. Let $\ax(g)$ be a $(\lambda,\epsilon)$-quasi-axis. By Morse Lemma, any two $(\lambda, \epsilon)$-quasi-axes of $g$ are within $L(\lambda, \epsilon,\delta)$ of each other. More generally, any sufficiently long  \revise{path} $J$ inside the $L(\lambda, \epsilon, \delta)$-neighborhood of $\ax(g)$ of $g$ has a natural orientation given by $g$: a point of $\ax(g)$ within $L(\lambda, \epsilon, \delta)$ of the terminal endpoint of $J$ is ahead (with respect to the $g$-orientation of $\ax(g)$) of a point of $\ax(g)$ within $L(\lambda, \epsilon, \delta)$ of the initial
endpoint of $J$ . We call this orientation of $J$ the \textit{$g$-orientation}.

\begin{definition}\label{Def: Equivalence}
Let $g_1$ and $g_2$ be two loxodromic elements of $G$.  We will write
$$
g_1 \sim g_2
$$
if there exists a constant $L'>0$ such that an arbitrarily long segment $J$ in $\ax(g_1)$  is contained in an $L'$-neighborhood of $t\ax(g_2)$ for some $t\in G$ and the map $t: J\to t(J)$ is orientation-preserving with respect to the $g_1$-orientation on $J$ and the $g_2$-orientation on $t(J)$.

\end{definition}

Note that $\sim$ is an equivalence relation. The following lemma gives a relation between the above definition and the definition (cf. Definition \ref{Def: Barriers}) of barriers which has nothing to do with the orientation.
\begin{lemma}\label{Lem: NotEquiImplyBarFree}
    If $g_1\nsim g_2^{\pm 1}$, then \revise{for any $\epsilon'>0$,} there exists $r>0$ such that $g_2^{\revise{m}}$ is $(\epsilon', g_1^{\revise{s}})$-barrier-free for any ${\revise{m}}\in \mathbb Z$ \revise{and $s\ge r$.} 
\end{lemma}
\begin{proof}
    Suppose to the contrary that \revise{there exists $\epsilon'>0$ such that} for every $r\in \mathbb N$, there \revise{exist $m\in \Z$ and $s\ge r$} such that $g_2^{\revise{m}}$ is not $(\epsilon', g_1^{\revise{s}})$-barrier-free. According to the definition of barriers, there exists $t\in G$ such that $d(to,[o,g_2^{\revise{m}}o])\le \epsilon'$ and $d(tg_1^{\revise{s}}o,[o,g_2^{\revise{m}}o])\le \epsilon'$. \rev{See Figure \ref{Fig8} for an illustration.}

    \begin{figure}[ht]
  \centering
  \begin{tikzpicture}
  \draw[fill=gray!10] (0,0) ellipse (6cm and 2cm);
\draw [black] (-4,0)--(4,0);
\filldraw [red] (-3,1) circle (1pt) node[black, above] {$to$};
\filldraw [red] (3,1) circle (1pt) node[black, above] {$tg_1^so$};
\filldraw [red] (4,0) circle (1pt) node[black, below] {$g_2^mo$};
\filldraw [red] (-4,0) circle (1pt) node[black, below] {$o$};
\draw [blue] (-3,1) .. controls (0,0.2) .. (3,1);
\draw [blue] (-3,1)--(-3,0);
\draw [blue] (3,1)--(3,0);
\filldraw [red] (3,0) circle (1pt) node[black, below] {$y$};
\filldraw [red] (-3,0) circle (1pt) node[black, below] {$x$};
\draw (-3,0.4) node[black,left] {$\leq \epsilon'$};
\draw (3,0.4) node[black,right] {$\leq \epsilon'$};
\draw (0,0.5) node[blue,above] {$\gamma$};
\draw (0,-1) node[black,above] {$N_{L'}([o,g_2^mo])$};
\end{tikzpicture}
  \caption{$[o,g_2^mo]$ contains an $(\epsilon',g_1^s)$-barrier; The blue path represents $\gamma$; The gray area represents $N_{L'}([o,g_2^mo])$}\label{Fig8}
\end{figure}

    Let $x,y\in [o,g_2^{\revise{m}}o]$ such that $d(to,x)=d(to,[o,g_2^{\revise{m}}o])$ and $d(tg_1^{\revise{s}}o,y)=d(tg_1^{\revise{s}}o,[o,g_2^{\revise{m}}o])$. Hence, the path $\gamma=[x,to]\cup [to,tg_1^{\revise{s}}o]\cup [tg_1^{\revise{s}}o,y]$ is a $(1,\revise{4}\epsilon')$-quasi-geodesic. By Lemma \ref{Lem: MorseLemma}, $t[o,g_1^{\revise{s}}o]\subset N_{L'}([o,g_2^{\revise{m}}o])$ \revise{where $L'=L(1,4\epsilon',\delta)$}. Up to exchanging $g_2$ to $g_2^{-1}$, we can assume that the map $t$ is orientation-preserving. Note that $s\ge r$. By Lemma \ref{Lem: StableLength}, one has $d(o,g_1^so)\ge s\|g_1\|\ge r\|g_1\|$, which implies that the length of $[o,g_1^so]\subseteq \ax(g_1)$ goes to infinity as $r\to \infty$. According to Definition \ref{Def: Equivalence}, this shows that $g_1\sim g_2$ or $g_1\sim g_2^{-1}$ which is impossible.

\end{proof}

\cite[Proposition 6]{BF02} shows that \revise{for every WPD action $G\curvearrowright X$,} there exist two  loxodromic elements $g_1$ and $g_2$ \revise{such that $g_1\nsim g_2$.}
\revise{In \cite[Proposition 2, Claim 1]{BF02}, Bestvina-Fujiwara proved that two elements $f_1, f_2$ satisfying the following constructions are non-equivalent: $$f_1=g_1^{n_1}g_2^{m_1}g_1^{k_1}g_2^{-l_1}, f_2=g_1^{n_2}g_2^{m_2}g_1^{k_2}g_2^{-l_2},$$ where $0\ll n_1\ll m_1\ll k_1\ll l_1\ll n_2\ll m_2\ll k_2\ll l_2$. But in fact, their proof only requires $0\ll n_1,m_1,|k_1|,l_1\ll n_2,m_2,|k_2|,l_2$, and does not require $k_1,k_2$ to be positive. So we can take $k_1,l_1, k_2,l_2$ to be $-n_1, m_1, -n_2, m_2$ respectively. In this case, $f_1=g_1^{n_1}g_2^{m_1}g_1^{-n_1}g_2^{-m_1}, f_2=g_1^{n_2}g_2^{m_2}g_1^{-n_2}g_2^{-m_2}\in [G,G]$ and the remaining proof of \cite[Proposition 2, Claim 1]{BF02} shows that }
\begin{lemma}\label{Lem: TwoIndependentElements}
There exist two  loxodromic elements $g_1$ and $g_2$ in $[G,G]$ on $X$ such that $g_1
\nsim g_2$.
\end{lemma}

Since $g_1$ and $g_2$ are independent, we may replace $g_1, g_2$ by high positive powers of conjugates to ensure that the subgroup $F$ of $G$ generated by $g_1, g_2$ is free with basis $S=\{g_1, g_2\}$, each nontrivial element of $F$ is loxodromic, and $F$ is quasi-convex with respect to the action on $X$ (see \revise{\cite[Proposition 4.3]{Fuj98}}). We will call such free subgroups \textit{Schottky groups}. Let $\mathcal G(F,S)$ be the Cayley graph of $F$ with respect to the generating set $S=\{g_1, g_2\}$. Then $\mathcal G(F,S)$ is a tree and each oriented edge has a label $g_i^{\pm 1}$ . Choose a base point $o\in X$ and construct an $F$-equivariant map $\Phi: \mathcal G(F,S)\to X$ that sends $1$ to $o$ and sends each edge to a geodesic arc. Quasi-convexity implies that $\Phi$ is a $(\lambda_0, \epsilon_0)$-quasi-isometric embedding for some $\lambda_0\ge 1, \epsilon_0\ge 0$ and in particular for every $1 \neq f \in F$ the $\Phi$-image of the axis of $f$ in $\mathcal G(F,S)$ is a $(\lambda_0, \epsilon_0)$-quasi-axis of $f$ in $X$.

Choose positive constants $$0\ll n_1\ll m_1\ll k_1\ll l_1\ll n_2\ll m_2\ll \cdots$$
and define $$f_i=g_1^{n_i}g_2^{m_i}g_1^{k_i}g_2^{-l_i}$$

for $i = 1, 2, 3, \ldots$.
\begin{proposition}\cite[Proposition 2]{BF02}\label{Prop: InfWords}
    $\{f_i:i\ge 1\}$ is an infinite sequence of loxodromic elements in $G$ such that
    \begin{enumerate}
        \item $f_i\nsim f_i^{-1}$ for $i\ge 1$, and
        \item $f_i\nsim f_j^{\pm 1}$ for $j<i$.
    \end{enumerate}
\end{proposition}

If $f \in F$ is cyclically reduced as a word in $g_1, g_2$ (equivalently, if its axis passes through $1 \in \mathcal G(F,S)$ ) then by the quasi-convexity of $F$ in $G$ we have
\begin{equation}\label{Quasi-convexity}
    d(o, f^\revise{m}(o)) \ge \revise{m}(d(o, f(o)) - 2L_1)
\end{equation}
where $L_1 = L(\lambda_0, \epsilon_0, \delta) > 0$ is given by Lemma \ref{Lem: MorseLemma} and is a constant independent of $f$ and $\revise{m}$.

\subsection{Constructing infinitely many quasimorphisms}
From now on, we fix an integer $W\ge 3L_1$ and will only consider \revise{a} path $w$ with $|w| > W$ . Thus, \revise{an almost realizing path $\alpha$ as in Lemma \ref{Lem: AlmostRealizing}}  will be a quasi-geodesic with constants independent
of $w$ and the endpoints. Moreover, $\alpha$ is contained in a uniform neighborhood, say, $L_2$-neighborhood, of any geodesic joining the endpoints of $\alpha$. We will also omit $W$
from the notation and write $c_w$ and $h_w$ for simplicity.

The next lemma is crucial for our discussion, not only in absolute bounded cohomology but also in the relative case. Recall the definition of barriers from Definition \ref{Def: Barriers}.

\begin{lemma} \label{Lem: no w}
Let $w=[o,fo]$ and $g\in G$ be an $(L_2,f)$-barrier-free element. Then we have $c_w(g)=0$ and $c_{w^{-1}}(g)=0$.
\end{lemma}

\begin{proof}
Assume that $c_w(g)>0$ and that $\alpha$ is \revise{an almost realizing path of $c_w(g)$}. From Lemma \revise{\ref{Lem: AlmostRealizing}} we know that $\alpha$ is a $(\frac{|w|}{|w|-W}, \frac{\revise{3}W|w|}{|w|-W})$-quasi-geodesic. Thus, $\alpha \subseteq N_{L_2}([o,go])$. Additionally, we have $d(o,go)-(|\alpha|-W|\alpha|_w)\revise{\ge c_w(g)/2}>0$. Therefore, $|\alpha|_w>(|\alpha|-d(o,go))/W> 0$. From the definition of $|\alpha|_w$, there exists some element $t\in G$ such that $t\cdot w\subseteq \alpha \subseteq N_{L_2}([o,go])$. This leads to a contradiction, as $g$ is $(L_2,f)$-barrier-free. Therefore, $c_w(g)=0$. We note that as long as there is no $t\in G$ such that $t \cdot w \subseteq N_{L_2}([o,go])$, there is also no $t\in G$ such that $t \cdot w^{-1} \subseteq N_{L_2}([o,go])$. The conclusion that $c_{w^{-1}}(g)=0$ then follows.
\end{proof}

For simplicity, for any $f\in G$, we set $c_f:=c_{[o,fo]}$ and $h_f:=h_{[o,fo]}$. Let $\{f_i:i\ge 1\}$ be the sequence from Proposition \ref{Prop: InfWords}. As the relation $\sim$ is invariant under conjugation, we assume in addition that each $f_i$ is cyclically reduced.

\begin{lemma}\label{Lem: UnboundedQM}
    For all $i\ge 1$, there exists $r_i>0$ such that \revise{for all $j<i$, we have}
    \begin{enumerate}
        \item  $h_{f_i^{r_i}}(f_i^{r_i\revise{m}})\ge L_1\revise{m}$ for any $\revise{m}\ge 0$.
        \item  $h_{f_i^{r_i}}$ is $0$ on $\langle f_j\rangle$.
    \end{enumerate}
\end{lemma}

\begin{proof}
    We first prove Item (2). Fix $i\ge 1$. As each element in the finite set $A:=\{f_1^{\pm 1}, \ldots, f_{i-1}^{\pm 1}\}$ is not equivalent to $f_i$, \revise{for any sufficiently large $\epsilon'>0$, }Lemma \ref{Lem: NotEquiImplyBarFree} gives a constant $r_i>0$ such that each $a\in A$ satisfies that $a^{\revise{m}}$ is $(\epsilon',f_i^{r_i})$-barrier-free for all ${\revise{m}}\in \mathbb Z$.

   \revise{Let $\epsilon'\ge L_2$.} As a result of Lemma \ref{Lem: no w}, we obtain
    $$c_{f_i^{r_i}}(a^{\revise{m}})=c_{f_i^{-r_i}}(a^{\revise{m}})=0$$
    for any $a\in A$ and $n\in \mathbb Z$.

    Hence, $h_{f_i^{r_i}}(f_j^{\revise{m}})=c_{f_i^{r_i}}(f_j^{\revise{m}})-c_{f_i^{-r_i}}(f_j^{\revise{m}})=0$ for all ${\revise{m}}\in \mathbb Z$.

    Now, we turn to proving Item (1).
    \begin{claim}
        For each $i\ge 1$, there exists $r_i'>0$ such that $c_{f_i^{\revise{r}}}(f_i^{-{\revise{rm}}})=0$ for any ${\revise{m}}\ge 0$ \revise{ and $r\ge r_i'$}.
    \end{claim}
    \begin{proof}[Proof of Claim]
        Suppose not. Then there exists $i\ge 1$ such that for every sufficiently large $r_i'>0$, there exists ${\revise{m}}\ge 0$ \revise{and $r\ge r_i'$} such that $c_{f_i^{\revise{r}}}(f_i^{-{\revise{rm}}})>0$. Denote $w=[o,f_i^{\revise{r}}o]$. Let $\alpha$ be \revise{an almost realizing path of} $c_w(f_i^{-{\revise{rm}}})$ in \revise{(\ref{Equ: AlmostReal})} with $f_i^{-1}$-orientation. Then we have $d(o,f_i^{-{\revise{rm}}}o)-(|\alpha|-W|\alpha|_w)\revise{\ge c_w(f_i^{-rm})/2}>0$. Thus, $|\alpha|_w>(|\alpha|-d(o,f_i^{-{\revise{rm}}}o))/W> 0$. From the definition of $|\alpha|_w$ there exists some element $t\in G$ such that \revise{$t\cdot w\subseteq \alpha\subseteq N_{L_2}([o,f_i^{-rm}o])$ and} the map $t$ respects the orientation. As $r_i'$ can be arbitrarily large, this implies that $f_i\sim f_i^{-1}$, which is a contradiction.
    \end{proof}
    Now we return to the proof of Item (1). For each $i\ge 1$, \revise{Lemma \ref{Lem: NotEquiImplyBarFree} allows us to require $r_i\ge r_i'$ where $r_i$ is the constant appearing in the proof of Item (2). Denote }$w=[o,f_i^{r_i}o]$. For $n\ge 1$, let $\gamma$ be the concatenated path $\cup_{0\le k \le {\revise{m}}-1}f_{i}^{r_ik}w$. According to (\ref{Quasi-convexity}), we have $$|\gamma|={\revise{m}}d(o, f_i^{r_i}o)\le d(o, f_i^{r_i {\revise{m}}}o)+2L_1 {\revise{m}}.$$ Obviously, $|\gamma|_w={\revise{m}}$. Recall that $W\ge 3L_1$. Then (\ref{Definition of c_w}) gives that $$ c_{f_i^{r_i}}(f_i^{r_i{\revise{m}}})\ge d(o, f_i^{r_i {\revise{m}}}o)-(|\gamma|-W|\gamma|_w)\ge d(o, f_i^{r_i {\revise{m}}}o)-(|\gamma|-3L_1 {\revise{m}})\ge L_1 {\revise{m}}.$$

    Therefore, $$h_{f_i^{r_i}}(f_i^{r_i{\revise{m}}})=c_{f_i^{r_i}}(f_i^{r_i{\revise{m}}})-c_{f_i^{-r_i}}(f_i^{r_i{\revise{m}}})= c_{f_i^{r_i}}(f_i^{r_i{\revise{m}}})-c_{f_i^{r_i}}(f_i^{-r_i{\revise{m}}})\ge L_1{\revise{m}}.$$
\end{proof}

Define $h_i : G \to \mathbb R$ as $h_i=h_{f_i^{r_i}}$, where $r_i>0$ is chosen as in Lemma \ref{Lem: UnboundedQM}, then we obtain the following.

\begin{proposition}\label{Prop: InfiniteQM}
 $\{h_i:i\ge 1\}$ is an infinite sequence of quasimorphisms on $G$ such that
\begin{enumerate}

\item $h_{i}(f_j^\revise{m})=0$ for all $i\neq j$ and for all $\revise{m}\geq 0$;
\item $h_{i}(f_i^{r_i\revise{m}})\geq L_1\revise{m}$ for all $i \geq 1$ and for all $\revise{m}\geq 0$;
\item $\psi(f_i)=0$ for all homomorphisms $\psi: G\rightarrow \mathbb{R}$;
\item \revise{the distance $d(o,f_io)$} tends to infinity as $i$ tends to infinity;
\item $h_{i}(h)=0$ for all $h\in H_j$ with $1\le j\le n$.
\end{enumerate}

\end{proposition}

\begin{proof}
The first two items follow directly from Lemma \ref{Lem: UnboundedQM}. Recall from Lemma \ref{Lem: TwoIndependentElements} that $g_1, g_2\in [G,G]$. Hence, each $f_i\in F=\langle g_1,g_2\rangle$ is a product of \revise{finitely many} commutators, which implies Item (3). \revise{Recall from the paragraph following Lemma \ref{Lem: TwoIndependentElements} that  $F=\langle g_1,g_2\rangle$ is a Schottky group which means that the orbital map $\Phi: \G(F,\{g_1,g_2\})\to X$ is a $(\lambda_0,\epsilon_0)$-quasi-isometric embedding. Since $f_i=g_1^{n_i}g_2^{m_i}g_1^{k_i}g_2^{-l_i}$ for $0\ll n_1,m_1,k_1,l_1,\ll n_2,m_2,k_2,l_2\ll \cdots$, one gets that $d(o,f_io)\ge \lambda_0^{-1}(n_i+m_i+k_i+l_i)-\epsilon_0$ which implies Item (4).} It suffices for us to verify Item (5).

As each $H_j(1\le j\le n)$ acts elliptically on $X$, there is a $D>0$ such that $d(o,ho)\le D$ for all $h\in \cup_{1\le j\le n}H_j$. \revise{Since $0\ll n_1,m_1,k_1,l_1\ll n_i,m_i,k_i,l_i$ for each $i\ge 2$, we can require $n_1\gg 0$ such that $d(o,f_i^{r_i}o)\ge \lambda_0^{-1}r_i(n_i+m_i+k_i+l_i)-\epsilon_0>2L_2+D$.} Then it follows from the definition of barriers (cf. Definition \ref{Def: Barriers}) that each $h$ is $(L_2,f_i^{r_i})$-barrier-free. Hence, as a result of Lemma \ref{Lem: no w}, $c_{f_i^{r_i}}(h)=c_{{f_i}^{-r_i}}(h)=0$. This shows that $h_{i}(h)=0$ for all $h\in \cup_{1\le j\le n}H_j$, which completes the proof.
\end{proof}

At the end of this section, we prove Proposition \ref{Prop: KeyProp}.

\begin{proof}[Proof of Proposition \ref{Prop: KeyProp}]
At first, we claim that for each $g\in G$ and $i\gg 0$, one has  $h_{i}(g)= 0$. Indeed, as Proposition \ref{Prop: InfiniteQM} (4) shows, \revise{$d(o,f_io)$} tends to infinity as $i\to\infty$. Hence, for $i\gg 0$, $g$ is $(L_2,f_i^{r_i})$-barrier-free \revise{since the diameter of $N_{L_2}([o,go])$ is finite. Thus one gets that}  $h_{i}(g)=0$ by Lemma \ref{Lem: no w}. Therefore, if $(a_i)\revise{_{i=1}^{\infty}} \in \ell^1$, then $\sum\limits_{i=1}^{\infty} a_i h_{i}$ is well-defined as an element of $C^1(G; \mathbb{R})$ since $\sum\limits_{i=1}^{\infty} a_i h_{i}(g)$ is in fact a finite sum for each $g\in G$. For the same reason, $\sum\limits_{i=1}^{\infty} a_i d^1 h_{i}$ is a well-defined 2-cocycle.   By Lemma \ref{Lem: UniformDefect}, all the 2-cocycles $d^1 h_{i}$ have a common bound \revise{, which means that there exists a constant $M>0$ such that $\sup_{g,g'\in G}|d^1 h_i(g,g')|\le \Delta(h_i)\le M$ for $i\ge 1$}. It follows that if a sequence $(a_i)\revise{_{i=1}^{\infty}}\in \ell^1$ then $\sum\limits_{i=1}^{\infty} a_i d^1 h_{i}$ is a bounded 2-cocycle. Therefore, we have the following equality.
 \begin{center}
 $\sum\limits_{i=1}^{\infty} a_i d^1 h_{i}=d^1(\sum\limits_{i=1}^{\infty}a_i h_{i})$.
 \end{center}

We get a real linear map $\omega: \ell^1 \rightarrow H^2_b(G; \mathbb{R})$ which sends the sequence $(a_i)\revise{_{i=1}^{\infty}}$ to the cohomology class represented by $\sum\limits_{i=1}^{\infty} a_i d^1 h_{i}$.  From Proposition \ref{Prop: InfiniteQM} (5), we know each $d^1 h_{i}$ lies in  $H^2_b(G, \{H_i\}_{i=1}^n; \mathbb R)$. So the real linear map is actually $\omega: \ell^1 \rightarrow  H^2_b(G, \{H_i\}_{i=1}^n; \mathbb R)$. In order to see that $\omega$ is injective, suppose $\omega((a_i))=0$. Then

\begin{center}
$d^1(\sum\limits_{i=1}^{\infty}a_i h_{i})=d^1 b$
\end{center}

\noindent for some bounded real-valued map $b\in C^1_b(G; \mathbb{R})$.  This means \revise{the following function}

\begin{center}
$\revise{\phi:=}\sum\limits_{i=1}^{\infty}a_i h_{i}-b$
\end{center}
\noindent \revise{is a homomorphism from $G$ to $\mathbb{R}$.} Apply this equality of 1-cochains to ${f_i}^{r_i\revise{m}} \in G$, we find

\begin{center}
$a_i h_{i}(f_i^{r_i\revise{m}})-b({f_i}^{r_i\revise{m}})=\phi({f_i}^{r_i\revise{m}})=0, \forall \revise{m}\ge 0.$
\end{center}

Since  $h_{i}(f_i^{r_i\revise{m}})\ge L_1\revise{m}$ and $b$ is a bounded map, this forces $a_i$ to be $0$. As $i$ is arbitrary, $(a_i)$ must be the zero vector. This shows the injectivity of $\omega: \ell^1\to H^2_b(G,\revise{\{H_i\}_{i=1}^n}; \mathbb R)$.

Finally, as $\ell^1$ has the dimension equal to the \revise{cardinality} of the continuum and the space of bounded cochains has cardinality $|\mathbb R^{\mathbb N}|=|\mathbb R|$, we complete the proof.
\end{proof}

As a consequence of Proposition \ref{Prop: KeyProp}, we have
\begin{proposition}\label{Prop: WPDQuotient}
    Let $G$ be a countable group and $H$ a normal subgroup of $G$. Suppose that $G/H$ acts WPD on a hyperbolic space. Then the dimension of $H^2_b(G, H; \mathbb R)$ as a vector space over $\mathbb R$ has the cardinality of the continuum.
\end{proposition}
\begin{proof}
    Proposition \ref{Prop: KeyProp} shows that the dimension of $H^2_b(G/H; \mathbb R)$ as a vector space over $\mathbb R$ has the cardinality of the continuum. Since Proposition \ref{Prop: NormalSubgp} shows that $H_b^2(G/H; \R)\cong H^2_b(G, H; \mathbb R)$, the conclusion follows.
\end{proof}

\section{The Proof of Theorem \ref{MainThm}}\label{Sec: MainThm}

Let us recall the conditions of Theorem \ref{MainThm}: $G$ acts properly on $X$ with contracting elements and $\{H_i\}(1\le i\le n)$ is a finite collection of Morse subgroups with infinite index in $G$.

\begin{definition}
    The action of a group $G$ on a metric space $X$ is \textit{acylindrical} if for all $L>0$ there exist $D>0$ and $B>0$ such that if $x,y\in X$ and $d(x,y)>D$, then there are at most $B$ elements $g\in G$ with $d(x,gx)\le L$ and $d(y,gy)\le L$.
\end{definition}
\revise{An acylindrical action $G\curvearrowright X$ is called \textit{non-elementary} if the action is unbounded and $G$ is not virtually cyclic. }

\begin{lemma}\label{Lem: AcyImpWPD}
    A non-elementary acylindrical action on a hyperbolic space must be WPD.
\end{lemma}
\begin{proof}
    Let $G$ be a group which admits a non-elementary acylindrical action on a hyperbolic space $X$.  Since $G\curvearrowright X$ is acylindrical, for all $L>0$, there exists $D>0$ such that if $x,y\in X$ and $d(x,y)>D$, then the set $\{g\in G: d(x,gx)\le L, d(y,gy)\le L\}$ is finite.

    Now we verify that $G\curvearrowright X$ is also a WPD action according to Definition \ref{Def: WPD}. By \cite[Theorem 1.1]{Osi16}, $G$ contains infinitely many independent loxodromic elements. Thus, it remains to verify the third item in Definition \ref{Def: WPD}. For every loxodromic element $g\in G$, every $x\in X$, and every $L>0$, let $D>0$ be the constant given by the above acylindrical action and $N>0$ be an integer depending only on $g$ such that $d(x,g^Nx)\ge N\|g\|>D$. Then the above acylindrical action implies that the set $\{h\in G: d(x,hx)\le L, d(g^Nx,hg^Nx)\le L\}$ is finite. This completes the verification.
\end{proof}

Fix a contracting element $g$ given by Lemma \ref{Lem: UniformShortProjection} and $K\gg 0$. Section \ref{Sec: MorseSubgpofInfIndex} produces a projection complex $\mathcal P_K(\mathcal F)$.

As shown in \cite[Theorem 1.5]{Sis18}, $E(g)$ is a hyperbolically embedded subgroup of $G$. Therefore, as a result of \cite[Theorem 5.6]{BBFS19}, $G$ acts acylindrically on $\mathcal P_K(\mathcal F)$. In particular, as a result of Lemma \ref{Lem: AcyImpWPD},

\begin{lemma}
    The action $G \curvearrowright \mathcal P_K(\mathcal F)$ satisfies WPD.
\end{lemma}

\begin{proof}[Proof of Theorem \ref{MainThm}]
    Recall that Lemma \ref{Lem: EllipticAction} shows that each $H_i(1\le i\le n)$ acts elliptically on  $\mathcal P_K(\mathcal F)$. Hence, the action $G \curvearrowright \mathcal P_K(\mathcal F)$ satisfies the setup of Section \ref{Sec: ConstructingQM}. Therefore, Theorem \ref{MainThm} follows from Proposition \ref{Prop: KeyProp}.
\end{proof}

A natural research direction to generalize our Theorem \ref{MainThm} is to assume instead that each subgroup $H_i$ is a subgroup with proper limit sets on a convergence boundary of $X$. See \cite{HYZ23, Yan22b} for more details about convergence boundary. A Morse subgroup with infinite index satisfies this property. Hence, one may wonder whether the following proposition is always true:
\begin{question}
    Let $G$ be a non-elementary countable  group acting properly on a geodesic metric space $X$ with convergence boundary. Let $H$ be a subgroup of $G$ with proper limit sets. Is the dimension of $H^{2}_b(G, H; \mathbb{R})$ as a vector space over $\mathbb{R}$  infinite?
\end{question}

At the end of this section, we provide an application of Theorem \ref{MainThm}.
\begin{definition}\label{Def: BoundedGeneration}
    A group $G$ is \textit{boundedly generated} by a finite collection of subgroups $H_1,\ldots, H_k$ if for every $g\in G$ there is a number $N$ such that all powers $g^n$ can be written in the form $$g^n=\prod_{i=1}^Nh_i(n)$$ where each $h_i(n)$ is conjugate to some element in $\cup_{1\le j\le k}H_j$.
\end{definition}
The above definition is a variation of \cite[Definition 4]{Kot04}. There Kotschick required each subgroup to be cyclic.
\begin{corollary}\label{MainCor}
Under the assumption of Theorem \ref{MainThm}, $G$ is not boundedly generated by $\{H_i: 1\le i\le n\}$.
\end{corollary}
\begin{proof}
    Suppose to the contrary that $G$ is boundedly generated by $\{H_i: 1\le i\le n\}$. Then for every $g\in G$, there exists an $N\in \mathbb N$ such that every power $g^m$ can be written as a product of $N$ elements in conjugations of $\cup_{1\le i\le n}H_i$.

    As Proposition \ref{Prop: InfiniteQM} shows, there is at least one unbounded quasimorphism $\phi$ on $G$ such that $\phi(h)=0$ for each $h\in H_i$ with $1\le i\le n$. Let $\bar \phi$ be the homogenization of $\phi$ given by Remark \ref{Rmk: HomogeneousQG}. Proposition \ref{Prop: InfiniteQM} also gives an element $g\in G$ such that $\bar \phi(g)>0$. Then there exists $N\in \mathbb N$ such that $g^m=h_1'\cdots h_N'$ for any $m>0$ and each $h_i'=g_ih_ig_i^{-1}$ with $g_i\in G, h_i\in \cup_{1\le j\le n}H_j$. Note that homogeneous quasimorphisms take constant values on conjugacy classes. Hence, $\bar \phi(h_i')=\bar \phi(h_i)=0$ for each $i$. Then one has that $$m|\bar \phi(g)|=|\bar \phi(g^m)|=|\bar \phi(h_1'\cdots h_N')|\le |\bar \phi(h_1'\cdots h_{N-1}')|+\Delta(\bar \phi)\le \cdots\le N\Delta(\bar \phi)$$ where $\Delta(\bar \phi)$ is the defect of $\bar \phi$.
    By letting $m\to \infty$, we reach a contradiction. Hence, we complete the proof.
\end{proof}

\section{Rotation family and relative bounded cohomology}\label{sec: RotationFamily}
In a group $G$, the normal closure of an element $g$ is denoted as $\langle\langle g\rangle\rangle$. The goal of this section is as follows:
\begin{proposition}\label{Prop: NormalClosure}
    Let $G$ be a non-elementary countable group acting properly on a geodesic metric space space $X$ with contracting elements. Then for any contracting element $g\in G$, there exists $k=k(g)>0$ such that the dimension of $H^{2}_b(G, \langle\langle g^k\rangle\rangle; \mathbb{R})$ as a vector space over $\mathbb{R}$ has the cardinality of the continuum.
\end{proposition}

\revise{
\noindent\textbf{Proof ideas of Proposition \ref{Prop: NormalClosure}:}
 Since $\llangle g^k\rrangle$ is  a normal subgroup of $G$, by Proposition \ref{Prop: WPDQuotient}, we only need to find a hyperbolic space such that the quotient group $G/\llangle g^k\rrangle$ acts acylindrically on it. According to the theory of rotating family developed by Dahmani-Guirardel-Osin \cite{DGO17}, we need to find a hyperbolic space such that $G$ acts acylindrically on it and $(\F=\{f\ax(g): f\in G\}, \{f\langle g^k\rangle f^{-1}: f\in G\})$ forms a rotating family. To obtain such a hyperbolic space, we need a construction of quasi-trees of spaces $\C(\F)$, which can be seen as a blow-up of the projection complex $\P_K(\F)$. \cite[Theorem 6.9]{BBFS19} shows that $\C(\F)$ is a quasi-tree on which $G$ acts acylindrically and $g$ is a loxodromic element. In order to get a suitable rotating family, we will consider a cone-off space $\Dot Z_r(\F)$ with apexes $\F$ over a scaled metric space $(\C(\F),l\cdot d_{\C})$. For $r \gg 0$, $\Dot{Z}_r(\F)$ is also hyperbolic \cite[Corollary 5.39]{DGO17}. Moreover, \cite[Lemma 5.3]{HLY20} shows that $(\F, \{f\langle g^k\rangle f^{-1}: f\in G\})$ is a suitable rotating family on $\Dot Z_r(\F)$. As a result of \cite[Proposition 5.28]{DGO17}, we get that $\Dot Z_r(\F)/\llangle g^k\rrangle$ is still hyperbolic. Finally, we verify that both the extended action $G\curvearrowright \Dot Z_r(\F)$ and the quotient action $G/\llangle g^k\rrangle\curvearrowright \Dot Z_r(\F)/\llangle g^k\rrangle$ are acylindrical.
}

\subsection{\revise{Quasi-trees} of spaces}
\revise{Fix a contracting element $g\in G$. We denote $\F=\{f\ax(g): f\in G\}$. Subsection \ref{subsec: ProjectionComplex} gives a projection complex $\P_K(\F)$ whose vertex set is $\F$ and two vertices $U,V\in \F$ are connected by an edge if and only if $\mathcal F_K(U,V) := \{W \in \mathcal F : d_W (U,V) > K\}=\emptyset$.} Fix a positive number $L$ such that $1/2K \le L \le 2K$. We now define a blowup
version, $\mathcal C(\mathcal F)$, of the projection complex $\mathcal P_K(\mathcal F)$ by preserving the geometry of each $U \in \mathcal F$. Namely,
we replace each $U \in \mathcal F$, a vertex in $\mathcal P_K(\mathcal F)$, with the corresponding subspace $U \subset X$, while maintaining the adjacency
relation in $\mathcal P_K(\mathcal F)$: if $U$ and  $V$ are adjacent in $\mathcal P_K(\mathcal F)$ (i.e., $d_{\mathcal P} (U, V ) = 1$), then we attach an edge of length
$L$ from every point $u \in \pi_U (V)$ to $v \in \pi_V (U)$. This choice of $L$, as stated in \cite[Lemma 4.2]{BBF15}, ensures that $U \subset X$ is geodesically embedded in $\mathcal C(\mathcal F)$ (so the index $L$ is omitted here).

For any contracting element $g\in G$, the infinite cyclic subgroup $\langle g\rangle$ is of finite index in $E(g)$ by \cite[Lemma 2.11]{Yan19}, so $\ax(g) = E(g)o$ is quasi-isometric to a line $\mathbb R$. Thus, the set $\mathcal F$ (derived from Lemma \ref{Lem: one element})  consists of uniform quasi-lines. By \cite[Theorem B]{BBF15}, we have the following:

\begin{theorem}\cite{BBF15}\label{Thm: C(F)}
    The quasi-tree of spaces $\mathcal C(\mathcal F)$ is a quasi-tree of infinite diameter, with each $ U \in \mathcal F$ totally geodesically embedded into $\mathcal C(\mathcal F)$. Moreover, the shortest projection from $U$ to $V$ in $\mathcal C(\mathcal F)$ agrees with the projection $\pi_U (V)$ up to a uniform finite Hausdorff distance.
\end{theorem}

Any two points $u \in U$ and $v \in V$ in $\mathcal C(\mathcal F)$ are connected via a standard path obtained from the standard path
$\alpha$ between $U$ and $V$ in $\mathcal P_K(\mathcal F)$, which passes through each vertex space $U$ on $\alpha$ via a geodesic in $U$ (see \cite[Definition 4.3]{BBF15} for more details). Hence, standard paths in $\mathcal C(\mathcal F)$ are also uniform quasi-geodesics.

\subsection{Hyperbolic cone-off and rotation family}
We first introduce a construction of a hyperbolic metric space by conning off a collection of
\revise{Morse} subsets from a given hyperbolic space.

Let $Z$ be a hyperbolic space with a collection $\mathcal F$ of uniformly Morse subsets, which means that these Morse subsets have a uniform Morse gauge \rev{(cf. Definition \ref{Def: MorseSubgp})}. Assume that $\mathcal F$ has bounded intersection \revise{(cf. Definition \ref{Def: BoundedIntersection})}. For $r \ge 0$, we first define the
hyperbolic cone-off $\Dot{Z}_r(\F)$ of $Z$ along $\F$.

For each $U\in \F$, the \textit{hyperbolic cone} $C_r(U)$ is the quotient space of the product
$U \times [0, r]$ by collapsing $U\times 0$. The collapsed point denoted by $a(U)$ is called the \textit{apex} of the cone , and $U \times 1$ the \textit{base} of the cone. The  cone is equipped with a geodesic metric such that it is the metric completion of the universal covering of a closed hyperbolic
disk punctured at the origin.

The \textit{hyperbolic cone-off} $\Dot{Z}_r(\F)$ is the quotient space of the disjoint union
$$Z\coprod_{U\in\F}C_r(U)$$ by gluing $U$ with the base of the cone $C_r(U)$, equipped with the length metric. Since $\F$ has bounded intersection, for $r \gg 0$, $\Dot{Z}_r(\F)$ is also hyperbolic \cite[Corollary 5.39]{DGO17}.

Assume that $G$ acts isometrically on $Z$ and leaves $\F$  invariant. The action naturally extends by isometry to the hyperbolic cone $C_r(U)$ by the rule $g(x, t) = (gx, t)$
for any $g \in G, x \in U, 0 \le t \le r$. This is a prototype of the notion of a rotating family
introduced in \cite{DGO17}.

\begin{definition}
    Assume $G$ acts isometrically on a metric space $\Dot{Z}$. Let $A$ be a $G$-invariant set in $\Dot Z$ and a collection of subgroups $\{G_a: a \in A\}$ of $G$ such that $G_a(a) = a, gG_ag^{-1} = G_{ga}$ for any $a \in A, g \in G$. We call such a pair $(A, \{G_a: a \in A\})$ a \textit{rotating family}.
\end{definition}

Returning to the above cone-off construction, the apexes $A(\F) = \{a(U): U\in \F\}$ and the stabilizers $G_a$ for $a \in A(\F)$ together consist of a rotating family. Moreover, we say that $A$ is \textit{$\rho$-separated} if any two distinct apexes are at distance at least $\rho$.

Roughly speaking, a rotating family $(A, \{G_a : a \in A\})$ is called \textit{very rotating}
if every nontrivial element in $G_a$ rotates around $a$ with a very large angle. This big angle is
usually achieved by taking a sufficiently deep subgroup (which is generated by a higher power of some element) of $G_a$.

\subsection{Proof of Proposition \ref{Prop: NormalClosure}}
From now on, we suppose that $G$ is a non-elementary countable group acting properly on a geodesic metric space $X$ with contracting elements. Fix a base point $o\in X$ and a contracting element $g\in G$. Denote $\mathcal F = \{f\ax(g) : f \in G\}$. Theorem \ref{Thm: C(F)} produces a quasi-tree of space $\mathcal C(\mathcal F)$ in which each $G$-translate of $\ax(g)$ is totally geodesically embedded.

\begin{lemma}\cite[Theorem 6.9]{BBFS19}\label{Lem: AcyActionOnC(F)}
    For $K\gg 0$, $\mathcal C(\mathcal F)$ is a quasi-tree on which $G$ acts acylindrically and $g$ is a loxodromic element on $\mathcal C(\mathcal F)$.
\end{lemma}

Denote by $d_{\mathcal C}$ the metric on $\mathcal C(\mathcal F)$. The following result gives a way to produce a very rotating family on some cone-off of a ``scaled'' quasi-tree of spaces. Here, a scaled metric means a constant multiple of the original metric. By scaling the metric of a hyperbolic space, one can require the hyperbolicity constant to be uniform.
\begin{lemma}\cite[Lemma 5.3]{HLY20}\label{Lem: RotFam}
    There exist universal constants $\delta_U > 0, r > 20\delta_U$ and $k=k(g), l=l(g)>0$ with the following property. Consider the cone-off space $\Dot{Z}_r(\mathcal F)$ with apexes $A(\F)$ over the scaled metric space $Z_l = (\C(\F),l\cdot d_{\C})$. For every $n\ge 1$, set $$E_n=\{f\langle g^{nk}\rangle f^{-1}:f\in G\}.$$ Then $(A(\F),E_n)$ is a $2r$-separated very rotating family on the $\delta_U$-hyperbolic space $\Dot{Z}_r(\mathcal F)$.
\end{lemma}
As $G$ acts acylindrically on $(\C(\F), d_{\C})$\revise{, it is straightforward to verify that $G$ also acts acylindrically on $(\C(\F), l\cdot d_{\C})$. Moreover,} \cite[Proposition 5.40]{DGO17} shows that acylindricity is preserved by taking (suitable) cone-off, \revise{thus} one has
\begin{lemma}
    The extended action $G\curvearrowright \Dot{Z}_r(\mathcal F)$ is acylindrical.
\end{lemma}
Fix $r> 10^{10}\delta_U$. As \cite[Proposition 5.33]{DGO17} also shows that acylindricity is preserved by taking quotient of a normal subgroup generated by a $2r$-separated very rotating family, one has
\begin{lemma}\label{Lem: QuoAcyAction}
    The quotient action $G/\llangle g^k\rrangle \curvearrowright \Dot{Z}_r(\mathcal F)/\llangle g^k\rrangle$ is acylindrical.
\end{lemma}
\begin{proof}
    In order to apply \cite[Proposition 5.33]{DGO17} to get the conclusion, we need to verify that there exists $K>0$ such that for all $a\in A(\F)$ and for all $x$ with $|x-a|=50\delta_U$, $\sharp\{h\in G: h(a)=a, |x-h(x)|\le 10\delta_U\}\le K$. From the construction of $A(\F)$ above, the stabilizer of each apex in $A(\F)$ is exactly a conjugate of $E(g)$. As $[E(g): \langle g\rangle]$ is finite, there exists $K=K(g)$ with the desired property.
\end{proof}

Moreover, the quotient space $\Dot{Z}_r(\mathcal F)/\llangle g^k\rrangle$ is $60000\delta_U$-hyperbolic by \cite[Proposition 5.28]{DGO17}.

\begin{proof}[Proof of Proposition \ref{Prop: NormalClosure}]
    As Lemma \ref{Lem: QuoAcyAction} implies that the quotient group acts WPD on a hyperbolic space, the conclusion follows from Proposition \ref{Prop: WPDQuotient}.
\end{proof}

In the end of this section, we posed the following question:
\begin{question}
     Let $G$ be a non-elementary countable  group acting properly on a geodesic metric space $X$ with contracting elements. Let $H$ be a normal subgroup of $G$ with a non-amenable quotient. Is the dimension of $H^{2}_b(G, H; \mathbb{R})$ as a vector space over $\mathbb{R}$  infinite?
\end{question}

\bibliographystyle{amsplain}
\bibliography{Reference}

\providecommand{\bysame}{\leavevmode\hbox to3em{\hrulefill}\thinspace}
\providecommand{\MR}{\relax\ifhmode\unskip\space\fi MR }
\providecommand{\MRhref}[2]{%
  \href{http://www.ams.org/mathscinet-getitem?mr=#1}{#2}
}
\providecommand{\href}[2]{#2}
\begin{thebibliography}{10}

\bibitem{ACGH16}
Goulnara~N Arzhantseva, Christopher~H Cashen, Dominik Gruber, and David Hume,
  \emph{Contracting geodesics in infinitely presented graphical small
  cancellation groups. preprint}, arXiv preprint arXiv:1602.03767 (2016).

\bibitem{Bal12}
Werner Ballmann, \emph{Lectures on spaces of nonpositive curvature}, vol.~25,
  Birkh{\"a}user, 2012.

\bibitem{BBF15}
Mladen Bestvina, Ken Bromberg, and Koji Fujiwara, \emph{Constructing group
  actions on quasi-trees and applications to mapping class groups}, Publ. Math.
  Inst. Hautes \'Etudes Sci. \textbf{122} (2015), 1--64. \MR{3415065}

\bibitem{BBFS19}
Mladen Bestvina, Ken Bromberg, Koji Fujiwara, and Alessandro Sisto,
  \emph{Acylindrical actions on projection complexes}, Enseign. Math.
  \textbf{65} (2019), no.~1-2, 1--32. \MR{4057354}

\bibitem{BF02}
Mladen Bestvina and Koji Fujiwara, \emph{Bounded cohomology of subgroups of
  mapping class groups}, Geom. Topol. \textbf{6} (2002), 69--89. \MR{1914565}

\bibitem{BF09}
\bysame, \emph{A characterization of higher rank symmetric spaces via bounded
  cohomology}, Geometric and Functional Analysis \textbf{19} (2009), no.~1,
  11--40.

\bibitem{Bou95}
Abdessalam Bouarich, \emph{Suites exactes en cohomologie born\'ee r\'eelle des
  groupes discrets}, C. R. Acad. Sci. Paris S\'er. I Math. \textbf{320} (1995),
  no.~11, 1355--1359. \MR{1338286}

\bibitem{BH99}
Martin~R. Bridson and Andr\'e Haefliger, \emph{Metric spaces of non-positive
  curvature}, Grundlehren der mathematischen Wissenschaften [Fundamental
  Principles of Mathematical Sciences], vol. 319, Springer-Verlag, Berlin,
  1999. \MR{1744486}

\bibitem{Bro81}
Robert Brooks, \emph{Some remarks on bounded cohomology}, Riemann surfaces and
  related topics: {P}roceedings of the 1978 {S}tony {B}rook {C}onference
  ({S}tate {U}niv. {N}ew {Y}ork, {S}tony {B}rook, {N}.{Y}., 1978), Ann. of
  Math. Stud., vol. No. 97, Princeton Univ. Press, Princeton, NJ, 1981,
  pp.~53--63. \MR{624804}

\bibitem{Bro94}
Kenneth~S. Brown, \emph{Cohomology of groups}, Graduate Texts in Mathematics,
  vol.~87, Springer-Verlag, New York, 1994, Corrected reprint of the 1982
  original. \MR{1324339}

\bibitem{BBFIPP14}
M.~Bucher, M.~Burger, R.~Frigerio, A.~Iozzi, C.~Pagliantini, and M.~B.
  Pozzetti, \emph{Isometric embeddings in bounded cohomology}, J. Topol. Anal.
  \textbf{6} (2014), no.~1, 1--25. \MR{3190136}

\bibitem{Cal09}
Danny Calegari, \emph{scl}, MSJ Memoirs, vol.~20, Mathematical Society of
  Japan, Tokyo, 2009. \MR{2527432}

\bibitem{CDP90}
M.~Coornaert, T.~Delzant, and A.~Papadopoulos, \emph{G\'eom\'etrie et th\'eorie
  des groupes}, Lecture Notes in Mathematics, vol. 1441, Springer-Verlag,
  Berlin, 1990, Les groupes hyperboliques de Gromov. [Gromov hyperbolic
  groups], With an English summary. \MR{1075994}

\bibitem{DGO17}
F.~Dahmani, V.~Guirardel, and D.~Osin, \emph{Hyperbolically embedded subgroups
  and rotating families in groups acting on hyperbolic spaces}, Mem. Amer.
  Math. Soc. \textbf{245} (2017), no.~1156, v+152. \MR{3589159}

\bibitem{Del96}
Thomas Delzant, \emph{Sous-groupes distingu\'es et quotients des groupes
  hyperboliques}, Duke Math. J. \textbf{83} (1996), no.~3, 661--682.
  \MR{1390660}

\bibitem{DK18}
Cornelia Dru\c~tu and Michael Kapovich, \emph{Geometric group theory}, American
  Mathematical Society Colloquium Publications, vol.~63, American Mathematical
  Society, Providence, RI, 2018, With an appendix by Bogdan Nica. \MR{3753580}

\bibitem{Dru05}
Cornelia Druţu and Mark Sapir, \emph{Tree-graded spaces and asymptotic cones
  of groups}, Topology \textbf{44} (2005), no.~5, 959--1058.

\bibitem{EF97}
David B.~A. Epstein and Koji Fujiwara, \emph{The second bounded cohomology of
  word-hyperbolic groups}, Topology \textbf{36} (1997), no.~6, 1275--1289.
  \MR{1452851}

\bibitem{Far12}
Benson Farb and Dan Margalit, \emph{A primer on mapping class groups (pms-49)},
  Princeton University Press, 2012.

\bibitem{Fra18}
Federico Franceschini, \emph{A characterization of relatively hyperbolic groups
  via bounded cohomology}, Groups Geom. Dyn. \textbf{12} (2018), no.~3,
  919--960. \MR{3845713}

\bibitem{FPS15}
R.~Frigerio, M.~B. Pozzetti, and A.~Sisto, \emph{Extending higher-dimensional
  quasi-cocycles}, J. Topol. \textbf{8} (2015), no.~4, 1123--1155. \MR{3431671}

\bibitem{Fri17}
Roberto Frigerio, \emph{Bounded cohomology of discrete groups}, Mathematical
  Surveys and Monographs, vol. 227, American Mathematical Society, Providence,
  RI, 2017. \MR{3726870}

\bibitem{Fuj98}
Koji Fujiwara, \emph{The second bounded cohomology of a group acting on a
  {G}romov-hyperbolic space}, Proc. London Math. Soc. (3) \textbf{76} (1998),
  no.~1, 70--94. \MR{1476898}

\bibitem{Fuj00}
\bysame, \emph{The second bounded cohomology of an amalgamated free product of
  groups}, Trans. Amer. Math. Soc. \textbf{352} (2000), no.~3, 1113--1129.
  \MR{1491864}

\bibitem{Ger15}
Victor Gerasimov and Leonid Potyagailo, \emph{Quasiconvexity in the relatively
  hyperbolic groups}, Journal f\"ur die reine und angewandte Mathematik (Crelle
  Journal) (2015).

\bibitem{Ghy90}
E.~Ghys and P.~da~la Harpe, \emph{Sur les groupes hyperboliques d’apr{\`e}s
  {Mikhael Gromov}}, Progress in Mathematics, Birkh{\"a}user Boston, 2013.

\bibitem{Gro82}
Michael Gromov, \emph{Volume and bounded cohomology}, Inst. Hautes \'Etudes
  Sci. Publ. Math. (1982), no.~56, 5--99. \MR{686042}

\bibitem{Gro87}
\bysame, \emph{Hyperbolic groups}, Essays in group theory, Math. Sci. Res.
  Inst. Publ., vol.~8, Springer, New York, 1987, pp.~75--263. \MR{919829}

\bibitem{Ham08}
Ursula Hamenst\"adt, \emph{Bounded cohomology and isometry groups of hyperbolic
  spaces}, J. Eur. Math. Soc. (JEMS) \textbf{10} (2008), no.~2, 315--349.
  \MR{2390326}

\bibitem{HYZ23}
Suzhen Han, Wenyuan Yang, and Yanqing Zou, \emph{Counting double cosets with
  application to generic 3-manifolds}, arXiv preprint arXiv:2307.06169 (2023).

\bibitem{HLY20}
Zunwu He, Jinsong Liu, and Wenyuan Yang, \emph{Large quotients of group actions
  with a contracting element}, Proceedings of the {I}nternational {C}onsortium
  of {C}hinese {M}athematicians 2017, Int. Press, Boston, MA, [2020]
  \copyright2020, pp.~319--338. \MR{4251117}

\bibitem{HO13}
Michael Hull and Denis Osin, \emph{Induced quasicocycles on groups with
  hyperbolically embedded subgroups}, Algebr. Geom. Topol. \textbf{13} (2013),
  no.~5, 2635--2665. \MR{3116299}

\bibitem{Iva85}
N.~V. Ivanov, \emph{Foundations of the theory of bounded cohomology}, vol. 143,
  1985, Studies in topology, V, pp.~69--109, 177--178. \MR{806562}

\bibitem{KK15}
Sungwoon Kim and Thilo Kuessner, \emph{Simplicial volume of compact manifolds
  with amenable boundary}, J. Topol. Anal. \textbf{7} (2015), no.~1, 23--46.
  \MR{3284388}

\bibitem{Kot04}
D.~Kotschick, \emph{Quasi-homomorphisms and stable lengths in mapping class
  groups}, Proc. Amer. Math. Soc. \textbf{132} (2004), no.~11, 3167--3175.
  \MR{2073290}

\bibitem{Min96}
Yair~N. Minsky, \emph{Quasi-projections in {Teichm\"{u}ller} space.}, Journal
  für die reine und angewandte Mathematik (Crelles Journal) \textbf{1996}
  (1996), no.~473, 121--136.

\bibitem{Mon01}
Nicolas Monod, \emph{Continuous bounded cohomology of locally compact groups},
  Lecture Notes in Mathematics, vol. 1758, Springer-Verlag, Berlin, 2001.
  \MR{1840942}

\bibitem{Neu54b}
B.~H. Neumann, \emph{Groups covered by permutable subsets}, J. London Math.
  Soc. \textbf{29} (1954), 236--248. \MR{62122}

\bibitem{Osi16}
D.~Osin, \emph{Acylindrically hyperbolic groups}, Trans. Amer. Math. Soc.
  \textbf{368} (2016), no.~2, 851--888. \MR{3430352}

\bibitem{Osi04}
Denis~V. Osin, \emph{Relatively hyperbolic groups: Intrinsic geometry,
  algebraic properties, and algorithmic problems}, Memoirs of the American
  Mathematical Society \textbf{179} (2004), no.~843.

\bibitem{PR15}
Cristina Pagliantini and Pascal Rolli, \emph{Relative second bounded cohomology
  of free groups}, Geom. Dedicata \textbf{175} (2015), 267--280. \MR{3323641}

\bibitem{Par03}
HeeSook Park, \emph{Relative bounded cohomology}, Topology Appl. \textbf{131}
  (2003), no.~3, 203--234. \MR{1983079}

\bibitem{Prz17}
Piotr Przytycki and Alessandro Sisto, \emph{A note on acylindrical
  hyperbolicity of mapping class groups}, Advanced Studies in Pure Mathematics,
  pp.~255--264, Mathematical Society of Japan, Japan, 2017 (English).

\bibitem{Sis18}
Alessandro Sisto, \emph{Contracting elements and random walks}, J. Reine Angew.
  Math. \textbf{742} (2018), 79--114. \MR{3849623}

\bibitem{Thu79}
W.P. Thurston, \emph{Travaux de {Thurston} sur les surfaces: S{\'e}minaire
  orsay}, Ast{\'e}risque, Vol. 66-67, Soci{\'e}t{\'e} math{\'e}matique de
  France, 1979.

\bibitem{WXY24}
Renxing Wan, Xiaoyu Xu, and Wenyuan Yang, \emph{Marked length spectrum rigidity
  in groups with contracting elements}, arXiv preprint arXiv:2402.10165 (2024).

\bibitem{Yan14}
Wenyuan Yang, \emph{Growth tightness for groups with contracting elements},
  Mathematical Proceedings of the Cambridge Philosophical Society \textbf{157}
  (2014), no.~2, 297–319.

\bibitem{Yan19}
\bysame, \emph{Statistically convex-cocompact actions of groups with
  contracting elements}, Int. Math. Res. Not. IMRN (2019), no.~23, 7259--7323.
  \MR{4039013}

\bibitem{Yan22b}
\bysame, \emph{Conformal dynamics at infinity for groups with contracting
  elements}, arXiv preprint arXiv:2208.04861 (2022).

\end{thebibliography}

\end{document}